\documentclass[10pt]{amsart}
\usepackage{amsmath}
\usepackage{tikz}

\usepackage{wrapfig} 
\usepackage{mathdots} 
\usepackage{enumitem}

\makeatletter

\def\normformnew#1[#2]{
\@ifnextchar({%
\decsuba[{#1}^{}_{#2}\left(]}%
{
{#1}^{}_{#2}}}
\def\decsuba[#1]#2{
\decsubb[#1]}

\def\decsubb[#1]#2,#3{
\@ifnextchar({
\decsuba[#1{}^{#2}_{#3}\left(]}%
{
\@ifnextchar){\decsubc[#1{}^{#2}_{#3}]}
{
\decsubd[#1]{#2}{#3}
}}}

\def\decsubc[#1]#2{
\@ifnextchar){
\decsubc[#1\right)]}%
{\@ifnextchar({
\decsuba[#1\right)\left(]}%
{
#1\right)}}}

\def\decsubd[#1]#2#3#4{
\@ifnextchar({
\decsuba[#1{}^{#2}_{#3#4}\left(]}%
{
\@ifnextchar){\decsubc[#1{}^{#2}_{#3#4}]}
{
\decsubd[#1]#2{#3#4}
}}}

\newcommand{\treenotation}[1]{\treenotationevenbetter{#1}}

\newcommand{\treenotationevenbetter}[1]{
\treenotationstart{#1}}

\def\treenotationstart#1{
\@ifnextchar[{
\treenotationsqsq{#1}}%
{
\@ifnextchar({
\treenotationsqro{#1}
}
{
#1}}}%

\def\hiddenindex{\phantom{1}}

\def\treenotationsqsq#1[#2]{%
\@ifnextchar[{
\treenotationsqsq{#1{}^{\hiddenindex}_{#2}}}%
{
\@ifnextchar({
\treenotationsqro{#1{}^{\hiddenindex}_{#2}}
}
{
#1{}^{\hiddenindex}_{#2}}}}%

\def\treenotationsqro#1(#2){%
\@ifnextchar[{
\treenotationrosq{#1{}^{#2}}}%
{
\@ifnextchar({
\treenotationroro{#1{}^{#2}_{\hiddenindex}}
}
{
#1{}_{\hiddenindex}^{#2}}}}%

\def\treenotationroro#1(#2){%
\@ifnextchar[{
\treenotationrosq{#1{}^{#2}}}%
{
\@ifnextchar({
\treenotationroro{#1{}^{#2}_{\hiddenindex}}
}
{
#1{}_{\hiddenindex}^{#2}}}}%

\def\treenotationrosq#1[#2]{%
\@ifnextchar[{
\treenotationsqsq{#1_{#2}}}%
{
\@ifnextchar({
\treenotationsqro{#1_{#2}}
}
{
#1_{#2}}}}%

\makeatother

\DeclareMathOperator{\opt}{opt}

\newcommand{\pspace}{\Omega}

\newcommand{\domlinprevs}{\mathcal{M}}

\newcommand{\decnode}{\treenotation{N}}
\newcommand{\chancenode}{\treenotation{N}}
\newcommand{\decision}{\treenotation{d}}
\newcommand{\event}{\treenotation{E}}
\newcommand{\reward}{\treenotation{r}}

\newcommand{\treeevent}[1]{\mathrm{ev}({#1})}
\newcommand{\rewardset}{\mathcal{R}}

\DeclareMathOperator{\back}{back}

\DeclareMathOperator{\backopt}{\back_{\opt}}

\newtheorem{lemma}{Lemma}

\newtheorem{corollary}[lemma]{Corollary}
\newtheorem{theorem}[lemma]{Theorem}
\newtheorem{example}[lemma]{Example}
\newtheorem{definition}[lemma]{Definition}

\newtheorem{property}{Property}

\newcommand{\tree}{T}
\newcommand{\atree}{U}

\newcommand{\setoftrees}{\mathcal{T}}

\DeclareMathOperator{\subtreeatoper}{st}
\newcommand{\subtreeat}[2]{\subtreeatoper_{#2}(#1)}

\DeclareMathOperator{\extoper}{ext}
\DeclareMathOperator{\normoper}{norm}
\DeclareMathOperator{\normgambles}{gamb}
\DeclareMathOperator{\children}{ch}
\DeclareMathOperator{\nfd}{nfd}

\newcommand{\biggambplus}{\bigoplus}
\newcommand{\gambplus}{\oplus}
\newcommand{\bigdecnodeunion}{\bigsqcup}
\newcommand{\decnodeunion}{\sqcup}
\newcommand{\bigchancenodemixture}{\bigodot}
\newcommand{\chancenodemixture}{\odot}

\newcommand{\compl}[1]{\overline{#1}}

\begin{document}

\title[]{Subtree Perfectness, Backward Induction, and Normal-Extensive Form Equivalence for Single Agent Sequential Decision Making under Arbitrary Choice Functions}
\author{Nathan Huntley}
\email{nath.huntley@gmail.com}
\address{Durham University, Department of Mathematical Sciences, Science Laboratories, South Road, Durham DH1 3LE, United Kingdom}
\author{Matthias C. M. Troffaes}
\email{matthias.troffaes@durham.ac.uk}
\address{Durham University, Department of Mathematical Sciences, Science Laboratories, South Road, Durham DH1 3LE, United Kingdom}


\date{16 September, 2011}

\begin{abstract}
  We revisit and reinterpret Selten's concept of subgame perfectness in the context of single agent normal form sequential decision making, which leads us to the concept of subtree perfectness. Thereby, we extend Hammond's characterization of extensive form consequentialist consistent behaviour norms to the normal form and to arbitrary choice functions under very few assumptions. In particular, we do not need to assume probabilities on any event or utilities on any reward. We show that subtree perfectness is equivalent to normal-extensive form equivalence, and is sufficient, but, perhaps surprisingly, not necessary, for backward induction to work. 
\end{abstract}

\maketitle

\section{Introduction}
\label{sec:intro}

In a single agent sequential decision problem, at any stage, one has two ways of looking at its solution: the problem can be considered either in its simplest form---discarding any past stages, or as part of a much larger problem---possibly considering choices and events that did not actually obtain. A reasonable requirement is that, at any stage, the solution is independent of the larger problem it is embedded in. In this paper, we call this requirement \emph{subtree perfectness}.

Selten~\cite{1975:selten} introduced a similar idea for multi-agent extensive form games, called \emph{subgame perfectness}. In such games, players decide sequentially, and a player's \emph{behaviour strategy} specifies a decision at each point where he must choose. An \emph{equilibrium point} is a strategy for each player, such that no player would change their strategy if they knew the strategy of all others.

A \emph{subgame} is a part of a game that is again a game. For example, if a game consists of four rounds, then after having played two rounds, one can consider the remaining two rounds as a separate game. An equilibrium point of the full game of course implies a strategy for each player in the subgame simply by restricting the strategy of each player in the full game to only that subgame. If, for every subgame, the restriction of the equilibrium point of the full game to that subgame also yields an equilibrium point of that subgame, then the equilibrium point of the full game is called \emph{subgame perfect}.

Selten showed that for games with perfect recall, (perfect\footnote{Without going into much detail, a \emph{perfect equilibrium point} is one which is stable under small perturbations \cite[p.~38]{1975:selten}.}) equilibrium points are subgame perfect \cite[p.~39, Thm.~2]{1975:selten}, that is, they are independent of any larger game in which they could be embedded.

As mentioned, in this paper we investigate single agent sequential decision making modelled by decision trees. Although such problems differ in many ways from extensive form games, subtree perfectness is clearly analogous to subgame perfectness, as the following example shows.

\begin{figure}
  \hfill
  \begin{minipage}[t]{.45\textwidth}
  \begin{center}
    \begin{tikzpicture}
      [minimum size=2em,parent anchor=east,child anchor=west,grow'=east,transform shape]
      \node[draw,rectangle]{$N_1$}
      [sibling distance=3em, level distance=4em]
      child{
        node[draw,rectangle]{$N_2$}
        [sibling distance=2em]
        child{
          node[right]{cake}
        }
        child{
          node[right]{ice cream}
        }
      }
      child{
        node[right]{scones}
      };
    \end{tikzpicture}
    \caption{Two-stage problem.}
    \label{fig:two:stage:problem}
  \end{center}
  \end{minipage}
  \hfill
  \begin{minipage}[t]{.45\textwidth}
  \begin{center}
    \begin{tikzpicture}[minimum size=2em,parent anchor=east,child anchor=west,grow'=east,transform shape]
      \node[draw,rectangle]{$N_2$}
      [sibling distance=2em]
      child{
        node[right]{cake}
      }
      child{
        node[right]{ice cream}
      };
    \end{tikzpicture}
    \caption{Second stage.}
    \label{fig:second:stage}
  \end{center}
  \end{minipage}
  \hfill
\end{figure}

Consider the decision problem in Fig.~\ref{fig:two:stage:problem}. In the first stage, the subject chooses between taking scones, or proceeding to the second stage. In the second stage, the subject chooses between cake, or ice cream. Suppose the subject prefers to reject scones and to choose ice cream at the second stage. This strategy induces a substrategy in the subtree for the second stage: choose ice cream over cake.

But, as with multi-agent games, we can instead consider the subtree for the second stage separately, as in Fig.~\ref{fig:second:stage}. If, in this smaller tree, the subject prefers ice cream, then his solution is \emph{subtree perfect}: his solution for the full tree induces a strategy in the subtree, and this strategy coincides with his solution for the subtree. If the subject states a difference preference (either no preference, or clear preference for cake), then his solution lacks subtree perfectness. So, subtree perfectness essentially means that the optimal induced strategies in a subtree do not depend on the full tree in which the subtree is embedded.

Our goal is to determine the conditions under which a theory of choice is subtree perfect. We consider a very large class of theories of choice, namely any that can be represented by conditional choice functions on gambles: we merely assume that for any set of gambles (functions from the possibility space $\Omega$ to a set of rewards $\rewardset$; these generalize random variables, or horse lotteries), and any conditioning event, the subject can give a non-empty subset of gambles that he considers optimal. Gambles that are non-optimal would never be selected, and the subject is unable to express further preference between the optimal ones. Maximizing expected utility is a simple example of such a choice function.

General choice functions, however, need neither probability nor utility---not even a preorder. Consequently, we must take care to distinguish clearly between the \emph{normal form} and the \emph{extensive form} of a decision problem (or game). 
While these forms are equivalent when maximizing expected utility \cite[1.3.4]{1961:raiffa}, their equivalence breaks down when expected utility is abandoned, as shown by for instance Seidenfeld \cite{1988:seidenfeld}, Machina \cite{1989:machina}, and Jaffray \cite{1999:jaffray}, among others.

Further, the standard definitions of normal and extensive form are not compatible with many choice functions. The usual extensive form solution for a decision tree, as given by Raiffa and Schlaifer~\cite[1.2.1]{1961:raiffa}, involves using backward induction to replace subtrees with their maximum expected utility. Clearly, this definition is too restrictive for general choice functions, particularly those who do not correspond to a total preorder. 
Therefore, instead, we use the terms normal and extensive form in a far more general sense, while retaining their core features: the normal form specifies all actions in all eventualities in advance, whereas the extensive form specifies actions locally at each decision node.

Our main results start out from the traditional normal form method of listing all strategies, finding their corresponding gambles, applying a choice function, and listing all strategies that induce optimal gambles. In doing so, we assume act-state independence---that is, choice functions do not depend on the decision. Dropping act-state independence is beyond the scope of this paper. Under this assumption, we find three necessary and sufficient conditions for a choice function to induce a subtree perfect normal form solution. Further, we show that a normal form solution induced by a choice function satisfying these conditions will always have an equivalent extensive form representation. Interestingly, this does not hold for subtree perfect normal form solutions in general, only those induced by a choice function.

Our results are very similar to those of Hammond~\cite{1988:hammond}, although it may appear that Hammond's goal is quite different from ours. Hammond considers extensive form solutions only, calls them \emph{behaviour norms} \cite[p.~28]{1988:hammond}, and defines a notion of \emph{consistency} which corresponds exactly to what we call subtree perfectness for extensive form solutions.
Two trees are called \emph{strategically equivalent} \cite[p.~1636]{1989:machina} 
if their normal form decisions (or, strategies) induce the same set of gambles. Hammond calls a behaviour norm \emph{consequentialist} if it preserves strategic equivalence, that is, if strategically equivalent trees also have strategically equivalent behaviour norms. Hammond shows that a consistent and consequentialist behaviour norm implies a choice function on gambles, and then investigates the properties of this choice function. Although our goal, to find necessary and sufficient conditions on \emph{arbitrary} choice functions under which they induce subtree perfect \emph{normal} form solutions, is apparently different, Hammond's approach effectively amounts to the same thing, but starting from extensive form, and only for a particular class of choice functions.

Even more so, because, as we will show, subtree perfectness implies normal-extensive form equivalence, unsurprisingly, Hammond's necessary and sufficient conditions on a choice function for it to be implied by a consistent and consequentialist behaviour norm, are very similar to our necessary and sufficient conditions for subtree perfectness. Hammond's first condition, that the choice function must induce a total preorder, is identical to our Property~\ref{prop:intersection:property}. His second, the independence axiom, is not present in our work because it only arises if probabilities can be assigned to certain chance arcs, an assumption we do not make. His third, a type of sure-thing principle, is very similar to our Property~\ref{prop:mixture:property}, where the difference is again based on whether or not probabilities are admitted. Finally, our Property~\ref{prop:conditioning:property} is not present in Hammond's account, because of a small difference in the definition of decision trees.

Since our study is closely related to Hammond's work, and makes use of many similar ideas, we shall throughout note where our definitions and properties coincide with those of Hammond, and also discuss the implications of the differences in approach. It turns out that the only significant difference is dependent upon exactly how consequentialism is defined. Hammond writes a description of consequentialism that turns out to be essentially identical to our approach, but his mathematical definition that he uses in all proofs is actually a weaker version, with which one cannot prove results quite as strong as ours. In Theorem~\ref{theorem:strong:hammond:theorem}, we unify our work with Hammond's informal but strong definition.


Subtree perfectness also relates to Machina's~\cite[p.~1627]{1989:machina} concept of \emph{separability over mutually exclusive events}, which essentially boils down to subtree perfectness for probability trees. Although Machina defines separability using probabilities and a total preorder, it can be easily generalized for arbitrary choice functions and more general gambles that do not involve probabilities. All such possible generalizations seem to be equivalent to, or implied by, our Property~\ref{prop:mixture:property}.

A different concept called \emph{separability} is defined by McClennen~\cite[p.~122]{1990:mcclennen}. Under the assumption of \emph{dynamic consistency} \cite[p.~120]{1990:mcclennen}, this form of separability proves to be equivalent to subtree perfectness. The condition of dynamic consistency is implicitly assumed in our solutions. Although once again McClennen's decision trees are somewhat different from ours, his Theorems 8.1 and 8.2 are very similar to our Lemma~\ref{lemma:conditioning:intersection:mixture:are:necessary:for:factuality}, and indeed the decision trees used in the proofs are almost identical.

Subtree perfectness is strongly linked to backward induction methods for solving decision trees. Indeed, under a total preorder, failures of backward induction are often used to illustrate violations of subtree perfectness~\cite{1986:lavalle,1989:machina}. We find that subtree perfectness is sufficient, but, perhaps surprisingly, \emph{not} necessary, for backward induction to work. In particular, a total preorder is not required. Further, backward induction implies a weaker form of subtree perfectness, in which every strategy must induce an optimal substrategy in every subtree, but the set of optimal strategies is not required to induce the set of all optimal substrategies in every subtree.

The paper is structured as follows: Section~\ref{sec:dectrees} explains decision trees and introduces notation. Section~\ref{sec:forms} provides a careful definition of normal and extensive form solutions, and introduces the concept of gambles to more easily work with normal form solutions. Section~\ref{sec:solution} introduces choice functions and their relationship with normal form solutions. Section~\ref{sec:counterfactuals} defines subtree perfectness and contains the principal results. Section~\ref{sec:backward:induction:and:factuality} explores when a backward induction method can be used to find the normal form solution induced by a choice function. Section~\ref{sec:discussion} examines the relationships subtree perfectness has with backward induction, extensive form equivalence, and the ideas of Hammond, McClennen, and Machina.

\section{Decision Trees}
\label{sec:dectrees}
In this section, we explain the basic ideas behind decision trees
and notation which will be used throughout the rest of the paper.
For more information about decision trees, we refer to the literature \cite{1985:lindley,2001:clemen}.

\subsection{Definition and Example}

A decision tree consists of a rooted tree \cite[p.~92, Sec.~3.2]{1999:gross::graph:theory} of decision nodes, chance nodes, and reward leaves, growing from
left to right. The left hand side corresponds to what happens first,
and the right hand side to what happens last. The payoffs for each
sequence of decisions and events are at the end.

Consider the following example. Tomorrow, a subject is going for
a walk in the lake district. It may rain ($E_1$), or
not ($E_2$). The subject can either take a waterproof ($d_1$), or
not ($d_2$), and can also choose to buy today's newspaper to learn tomorrow's
weather forecast ($d_S$), or not ($d_{\compl{S}}$). Suppose that the
forecast can have either two outcomes: predicting rain ($S_1$), or not
($S_2$).
The utility of each combination, if the subject does not buy the newspaper, is
summarized in Figure~\ref{fig:lake:tree:basic} (left).
If the subject buys the newspaper, then one utile is subtracted.

Figure~\ref{fig:lake:tree:basic} also depicts the decision tree.
\begin{figure}
\begin{minipage}[t]{0.25\linewidth}
\begin{center}
  {\small
  \begin{tabular}{r|cc}
    & $E_1$ & $E_2$ \\
    \hline
    $d_1$ & $10$ & $15$ \\
    $d_2$ & $5$ & $20$
  \end{tabular}
  }
\end{center}
\end{minipage}
\hfill
\begin{minipage}{0.7\linewidth}
  \begin{center}
    \begin{tikzpicture}
      [minimum size=2em,parent anchor=east,child anchor=west,grow'=east,scale=0.75,transform shape]
        \node[draw,rectangle]{$\decnode[1]$}
        [sibling distance=8em,level distance=5em]
        child{
          node[draw,circle]{$\chancenode[1](1)$}
          [sibling distance=8em,level distance=12em]
          child{
            node[draw,rectangle]{$\decnode[1](1)[1]$}
            [sibling distance=4em,level distance=5em]
            child{
              node[draw,circle]{$\chancenode[1](1)[1](1)$}
              [sibling distance=1.2em,level distance=5em]
              child{
                node{$9$}
                edge from parent
                node[above,sloped]{$E_1$}
              }
              child{
                node{$14$}
                edge from parent
                node[below,sloped]{$E_2$}
              }
              edge from parent
              node[above,sloped]{$d_1$}
            }
            child{
              node[draw,circle]{$\chancenode[1](1)[1](2)$}
              [sibling distance=1.2em,level distance=5em]
              child{
                node{$4$}
                edge from parent
                node[above,sloped]{$E_1$}
              }
              child{
                node{$19$}
                edge from parent
                node[below,sloped]{$E_2$}
              }
              edge from parent
              node[below,sloped]{$d_2$}
            }
            edge from parent
            node[above,sloped]{$S_1$}
          }
          child{
            node[draw,rectangle]{$\decnode[1](1)[2]$}
            [sibling distance=4em,level distance=5em]
            child{
              node[draw,circle]{$\chancenode[1](1)[2](1)$}
              [sibling distance=1.2em,level distance=5em]
              child{
                node{$9$}
                edge from parent
                node[above,sloped]{$E_1$}
              }
              child{
                node{$14$}
                edge from parent
                node[below,sloped]{$E_2$}
              }
              edge from parent
              node[above,sloped]{$d_1$}
            }
            child{
              node[draw,circle]{$\chancenode[1](1)[2](2)$}
              [sibling distance=1.2em,level distance=5em]
              child{
                node{$4$}
                edge from parent
                node[above,sloped]{$E_1$}
              }
              child{
                node{$19$}
                edge from parent
                node[below,sloped]{$E_2$}
              }
              edge from parent
              node[below,sloped]{$d_2$}
            }
            edge from parent
            node[below,sloped]{$S_2$}
          }
          edge from parent
          node[above,sloped]{$d_S$}
        }
        child{
          node[draw,rectangle]{$\decnode[1][2]$}
          [sibling distance=4em,level distance=5em]
          child{
            node[draw,circle]{$\chancenode[1][2](1)$}
            [sibling distance=1.2em,level distance=5em]
            child{
              node{$10$}
              edge from parent
              node[above,sloped]{$E_1$}
            }
            child{
              node{$15$}
              edge from parent
              node[below,sloped]{$E_2$}
            }
            edge from parent
            node[above,sloped]{$d_1$}
          }
          child{
            node[draw,circle]{$\chancenode[1][2](2)$}
            [sibling distance=1.2em,level distance=5em]
            child{
              node{$5$}
              edge from parent
              node[above,sloped]{$E_1$}
            }
            child{
              node{$20$}
              edge from parent
              node[below,sloped]{$E_2$}
            }
            edge from parent
            node[below,sloped]{$d_2$}
          }
          edge from parent
          node[below,sloped]{$d_{\compl{S}}$}
        };
    \end{tikzpicture}
  \end{center}
\end{minipage}
    \caption{Payoff table and decision tree for walking in the lake district.}
  \label{fig:lake:tree:basic}
\end{figure}
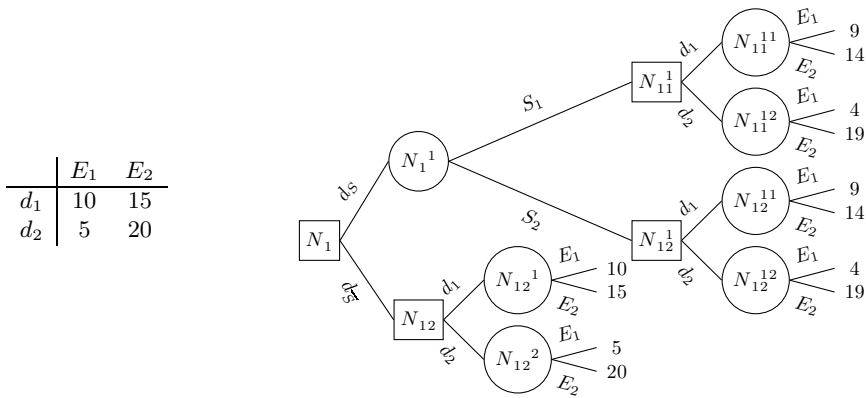
Squares are decision nodes, and circles are chance nodes.
From each node, branches emerge. For decision
nodes, these are decisions; for chance nodes, these are events. For each chance node, the events form a partition of the possibility space, contrary to Hammond~\cite[p.~31]{1988:hammond}, where events at chance nodes form a partition of the set of all states of nature still possible upon reaching the chance node. 
This leads to several technical differences, that will be seen in Definition~\ref{def:consistent:decision:trees}, Property~\ref{prop:conditioning:property}, and Property~\ref{prop:backward:conditioning:property}. Furthermore, our chance nodes are called \emph{natural nodes} by Hammond. Hammond's chance nodes involve probabilities and do not appear in our decision trees.

In our example, the
subject must first decide whether to buy the newspaper or
not, hence we start with a decision node. If he
buys the newspaper ($d_S$), then he learns about tomorrow's
forecast. Thus, the chance node following $d_S$
has two branches, forecasting rain ($S_1$), or no rain
($S_2$). Next, when leaving for the trip, he can
either bring his waterproof ($d_1$) or not ($d_2$), hence the
decision node following $S_1$. During the walk, a chance
node yields rain ($E_1$) or not ($E_2$).

So, each path in a decision tree amounts to a particular sequence
of decisions and events. The payoffs resulting from each such
sequence are put at the end.

\subsection{Notation}

Before elaborating how decision trees can be solved, we introduce a convenient mathematical notation for them. Decision trees can be seen as combinations of smaller decision trees: for instance, in the lake district example, one could draw the subtree corresponding to buying the newspaper, and also draw the subtree corresponding to making an immediate decision. The decision tree for the full problem is then formed by joining these two subtrees at a decision node.

Hence, we can represent a decision tree as follows. 
Let $\tree_1$, \dots, $\tree_n$ be decision trees and $\event[1]$, \dots, $\event[n]$ be a partition of the possibility space. If $\tree$ is formed by combining the trees at a decision node, we write
\begin{equation*}
  \tree = \bigdecnodeunion_{i=1}^n \tree_i.
\end{equation*}
If $\tree$ is formed by combining the trees at a chance node, with subtree $\tree_i$ being connected by event $\event[i]$, we write
\begin{equation*}
  \tree = \bigchancenodemixture_{i=1}^n \event[i] \tree_i.
\end{equation*}
For instance, for the tree of Fig.~\ref{fig:lake:tree:basic}, we write
\begin{equation*}
  (S_1(T_1\decnodeunion T_2)\chancenodemixture S_2(T_1\decnodeunion T_2))\decnodeunion(U_1\decnodeunion U_2)
\end{equation*}
with
\begin{align*}
  T_1&=E_1 9 \chancenodemixture E_2 14
  &
  U_1&=E_1 10 \chancenodemixture E_2 15
  &
  T_2&=E_1 4 \chancenodemixture E_2 19
  &
  U_2&=E_1 5 \chancenodemixture E_2 20
\end{align*}
In certain circumstances it will be necessary to consider decision trees formed by adding a decision node in front of a tree $\tree$, that is, a decision tree whose root is a decision node with one option. Such a tree shall be denoted by $\decnodeunion \tree$.

In this paper we shall often be considering subtrees of larger trees. For subtrees, we need to know the events that were observed in the past. Two subtrees with the same configuration of nodes and arcs may have different preceding events, and should be treated differently. Therefore we associate with every decision tree $\tree$ an event $\treeevent{\tree}$: the intersection of all the events on chance arcs that have preceded $\tree$. Hammond~\cite[p.~27]{1988:hammond} denotes these events by $S(n)$.

\begin{definition}\label{def:subtreeat}
 A subtree of a tree $\tree$ obtained by removal of all non-descendants of a particular node $N$ is called \emph{the subtree of $\tree$ at $N$} and is denoted by $\subtreeat{\tree}{N}$.
\end{definition}

Subtrees are called \emph{continuation trees} by Hammond \cite{1988:hammond}.

\section{Solving Decision Trees}
\label{sec:forms}

This paper deals with more general solutions of decision trees than are usually considered. Consequently, the usual definitions of extensive and normal forms, such as in Raiffa and Schlaifer \cite{1961:raiffa}, are insufficient. Therefore, we first carefully define normal and extensive form solutions.

\subsection{Extensive and Normal Form Solutions}
\label{sec:extensive:normal:form:solutions}

The usual definition of extensive form is based on backward induction and expected utility. At each ultimate decision node, we calculate the expected utility of each option, and choose a maximal arc, replacing that node by its maximum expected utility. The penultimate decision nodes have now become the ultimate ones, so we can repeat this process, until the root node.

If options are not assigned values, 
then this type of backward induction cannot be used. 
Therefore, we abandon the link with backward induction, and instead focus on another property of extensive form: the decision arc to follow only needs to be specified when the subject actually reaches the decision node.

An \emph{extensive form solution} of a decision tree removes from each decision node some (possibly none), but not all, of the decision arcs. So, an extensive form solution is a subtree of the original decision tree, where at each decision node a non-empty subset of arcs is retained. For instance, in the lake district example, one of the extensive form solutions is: do not buy the newspaper, and then either take the waterproof or not. An extensive form solution can be used as follows: the subject, upon reaching a decision node, chooses one of the arcs in the extensive form solution, and follows it. The subject only needs to decide which arc to follow at a decision node when reaching that node.


In contrast, normal form solutions specify all future decisions at the start, so no further decisions need to be taken as time progresses. Classically, normal form solutions involve listing all possible combinations of actions and then choosing one to maximize expected utility. This generalizes 
as follows.

First, an extensive form solution with just one arc out of each decision node, is called a \emph{normal form decision}. For instance, in the example, one of the normal form decisions is: buy the newspaper, and take the waterproof if the newspaper predicts rain, but do not take the waterproof otherwise. We denote the set of all normal form decisions for a decision tree $\tree$ by $\nfd(\tree)$. Normal form decisions are also called \emph{strategies}, \emph{pure strategies}, \emph{plans}, and \emph{policies}.


A \emph{normal form solution} of a decision tree $\tree$ is then simply a subset of $\nfd(\tree)$. The interpretation of this subset is that the subject picks one of the normal form decisions of the normal form solution, and then acts accordingly.

It is natural to ask whether solutions of one form can be meaningfully transformed into solutions of the other. This question is addressed in Section~\ref{sec:equivalence}. For now, note only that there are usually more normal form solutions than there are extensive form solutions. For example, consider the tree below,
\begin{wrapfigure}[7]{r}{0.2\linewidth}
  \begin{center}
    \begin{tikzpicture}
      [minimum size=2em,parent anchor=east,child anchor=west,grow'=east,scale=0.75,transform shape]
	\node[draw,circle]{}
	[sibling distance=4em, level distance=4em]
	child{
		node[draw,rectangle]{}
	[sibling distance=1em, level distance=4em]
		child{
			node{$\reward(1)[1]$}
			edge from parent
			node[above,sloped]{$\decision(1)[1]$}
		}
		child{
			node{$\reward(1)[2]$}
			edge from parent
			node[below,sloped]{$\decision(1)[2]$}
		}
		edge from parent
		node[above,sloped]{}
	}
	child{
        node[draw,rectangle]{}
        [sibling distance=1em, level distance=4em]
		child{
			node{$\reward(2)[1]$}
			edge from parent
			node[above,sloped]{$\decision(2)[1]$}
		}
		child{
			node{$\reward(2)[2]$}
			edge from parent
			node[below,sloped]{$\decision(2)[2]$}
		}
		edge from parent
		node[below,sloped]{}
	};
\end{tikzpicture}
\end{center}
\end{wrapfigure}
and suppose a normal form solution contains two normal form decisions, namely $\decision(1)[1]\decision(2)[2]$ and $\decision(1)[2]\decision(2)[1]$. Any attempt to find a corresponding extensive form solution will have to include all four decision arcs, but that would also correspond to the normal form solution where all four normal form decisions are present. Therefore there is no one-to-one correspondence between extensive and normal form solutions.

These two forms of solution are not the only possibilities. For instance, McClennen \cite{1990:mcclennen} considers a form of solution where the subject must give a set of optimal normal form decisions \emph{at every node in the tree}. The interpretation is apparently that the subject chooses a plan every time he reaches a node. We shall call such a solution a \emph{dynamic normal form} solution. Our normal form solutions are special cases of dynamic normal form solutions, where the plan at any node is simply the restriction, to the node in question, of the plan at the root node (McClennen calls this a \emph{dynamically consistent} solution).

\subsection{Extensive and Normal Form Operators}
\label{sec:operators}

An \emph{extensive form operator} is a function which maps each decision tree to an extensive form solution of that decision tree. The method by which decision arcs are removed is not specified, and in particular, need not be related to backward induction. Hammond~\cite[p.~28]{1988:hammond} calls these operators \emph{behaviour norms}.

A \emph{normal form operator} maps each decision tree to a normal form solution of that tree. Again, the method by which this happens is not part of our definition.

These operators are usually 
interpreted as describing optimality.

Clearly, the classical extensive and normal form interpretations are, respectively, extensive and normal form operators. Classical backward induction deletes any arc with non-maximal expectation, and results in a tree with some arcs deleted: an extensive form solution. The classical normal form finds the set of normal form decisions with maximal expectation: a normal form solution.

\subsection{Gambles}

To express normal form decisions and solutions efficiently, we first introduce some definitions and notation. Let $\pspace$ be the \emph{possibility space}: the set of all possible states of the world. We only consider finite possibility spaces. Elements of
$\pspace$ are called typically denoted by $\omega$.
Subsets of $\pspace$ are called \emph{events}.
Let $\rewardset$ be a set of rewards. For our results we need \emph{not} assume $\rewardset=\mathbb{R}$. 

A \emph{gamble} is a function $X\colon \pspace\to\rewardset$, and is interpreted as an uncertain reward: once $\omega\in\pspace$ is observed, $X$ yields $X(\omega)$. Probabilities over $\pspace$ are \emph{not} needed.

\subsection{Normal Form Gambles}
\label{sec:normal:form:gambles}

Once a normal form decision is chosen, the reward is determined entirely by the events that obtain. So, every normal form decision has a corresponding gamble, which we call a \emph{normal form gamble}. The set of all normal form gambles associated with a decision tree $\tree$ is denoted by $\normgambles(\tree)$.

Using Fig.~\ref{fig:lake:tree:basic}, we explain how to find normal form gambles. First consider the subtree at $\decnode[1](1)[1]$, which has normal form decisions $d_1$ and $d_2$. The former gives reward $9$ utiles if $\omega\in E_1$ and $14$ utiles if $\omega\in E_2$, and so yields the gamble
\begin{equation}\label{eq:gambleexample}
E_1 9 \gambplus E_2 14.
\end{equation}
The $\gambplus$ operator combines partial maps defined on disjoint domains (i.e. the partial map $E_1 9$ defined on $E_1$, and the partial map $E_2 14$ defined on $E_2$).

Now consider the subtree with root at $\chancenode[1](1)$, and in particular the normal form decision `$d_1$ if $S_1$ and $d_2$ if $S_2$'. This gives reward $9$ if $\omega \in S_1 \cap E_1$, reward $14$ if $\omega \in S_1 \cap E_2$, and so on. The corresponding gamble is
\begin{equation*}
(S_1 \cap E_1) 9 \gambplus (S_1 \cap E_2) 14
\gambplus
(S_2 \cap E_1) 4 \gambplus (S_2 \cap E_2) 19,
\end{equation*}
or briefly, if we omit `$\cap$' and employ distributivity,
\begin{equation}\label{eq:normal:form:gamble:example}
S_1 \left( E_1 9 \gambplus E_2 14 \right)
\gambplus S_2 \left( E_1 4 \gambplus E_2 19 \right),
\end{equation}
where multiplication with an event is now understood to correspond to restriction, i.e., $9$ is a constant map on $\pspace$, $E_1 9$ is a constant map restricted to $E_1$, and $S_1 (E_1 9)$ is obtained from $E_1 9$ by further restriction to $E_1\cap S_1$. For illustration, we tabulate the values of some normal form gambles in Table~\ref{tab:normal:form:gambles:example}, where $\pspace=\{\omega_1, \omega_2, \omega_3, \omega_4\}$, $E_1 = \{\omega_1, \omega_2\}$, and $S_1 = \{\omega_1, \omega_3\}$.

\begin{table}
  \begin{center}
    {\small
    \begin{tabular}{r|cccc}
      & $\omega_1$ & $\omega_2$ & $\omega_3$ & $\omega_4$ \\
      \hline
      $E_1 9 \gambplus E_2 14$ & $9$ & $9$ & $14$ & $14$ \\
      $S_1 \left( E_1 9 \gambplus E_2 14 \right)
\gambplus S_2 \left( E_1 4 \gambplus E_2 19 \right)$ & $9$ & $4$ & $14$ & $19$
    \end{tabular}
    }
    \caption{Example of normal form gambles.}
    \label{tab:normal:form:gambles:example}
  \end{center}
\end{table}

Observe that the gamble in Eq.~\eqref{eq:normal:form:gamble:example} incorporates the gamble in Eq.~\eqref{eq:gambleexample}. 
This observation allows a very convenient recursive definition of the $\normgambles$ operator. 

\begin{definition}\label{def:gambsum:for:sets}
 For any events $\event_1$, \dots, $\event_n$ which form a partition, and any finite family of sets of gambles $\mathcal{X}_1$, \dots, $\mathcal{X}_n$, we define the following set of gambles:
\begin{equation}
 \label{eq:def:setsum}
 \biggambplus_{i=1}^n \event_i\mathcal{X}_i
 =
 \left\{\biggambplus_{i=1}^n \event_i X_i\colon X_i\in\mathcal{X}_i\right\}
\end{equation}
\end{definition}


\begin{definition}
 \label{def:normgambles}
 With any decision tree $\tree$, we associate a set of gambles $\normgambles(\tree)$, recursively defined through:
\begin{subequations}
\begin{itemize}
\item If a tree $\tree$ consists of only a leaf with reward $\reward\in\rewardset$, then
 \begin{equation}\label{eq:normgambles:rewards}
   \normgambles(\tree)=\{\reward\}.
 \end{equation}
\item If a tree $\tree$ has a chance node as root, that is, $\tree=\bigchancenodemixture_{i=1}^n E_i \tree_i$, then
\begin{equation}\label{eq:normgambles:chancenodes}
 \normgambles\left(\bigchancenodemixture_{i=1}^n E_i \tree_i\right)
 =
 \biggambplus_{i=1}^n \event_i\normgambles(\tree_i).
\end{equation}
\item If a tree $\tree$ has a decision node as root, that is, if $\tree=\bigdecnodeunion_{i=1}^n \tree_i$, then
\begin{equation}\label{eq:normgambles:decisionnodes}
 \normgambles\left(\bigdecnodeunion_{i=1}^n \tree_i\right)
 =
 \bigcup_{i=1}^n \normgambles(\tree_i).
\end{equation}
\end{itemize}
\end{subequations}
\end{definition}

It is easy to show that this definition gives indeed the required set,
\begin{equation}\label{eq:normgambles:tree:is:normgambles:nfd:tree}
  \normgambles(\tree) = \bigcup_{U\in\nfd(\tree)}\normgambles(U).
\end{equation}
Hammond uses $F(T,n)$ for $\normgambles(\subtreeat{\tree}{N})$. McClennen uses $G(\tree)$ for $\normgambles(\tree)$.

Multiple decision trees can model the same problem. This suggests the following definition (see for instance \cite{1987:lavalle,1989:machina}):

\begin{definition}
 Two decision trees $\tree_1$ and $\tree_2$ are called \emph{strategically equivalent} if $\normgambles(\tree_1)=\normgambles(\tree_2)$.
\end{definition}
Hammond~\cite[p.~38]{1988:hammond} uses the term \emph{consequentially equivalent}. 

\section{Normal Form Solutions for Decision Trees}
\label{sec:solution}

\subsection{Choice Functions and Optimality}
\label{sec:choice:functions}

We defined a normal form solution as a subset of all normal form decisions. Ideally one would like to identify a single best normal form decision, but this is not always possible. The subject might, however, still be able to eliminate some normal form decisions that he would never consider choosing, leaving a subset of normal form decisions. We say that the subject considers these as \emph{optimal}.

In classical decision theory, each normal form decision induces a random real-valued gain, and is considered optimal if its expected gain is maximized. As another example, suppose that 
a set $\domlinprevs$ of plausible probability distributions are specified. Then the subject might consider optimal all normal form decisions whose expected gain is maximal under at least one distribution in $\domlinprevs$.

So, often, optimal decisions are determined by comparison of gambles. We follow this common approach,
 since normal form decisions have corresponding gambles, and gambles are easier to work with. We therefore suppose that the subject can determine an optimal subset of any set of gambles, conditional upon an event $A$ (corresponding to $\treeevent{\tree}$ of the tree $\tree$ in question):

\begin{definition}\label{def:choice:function}
  A choice function $\opt$ maps, for any non-empty event $A$, each non-empty finite set $\mathcal{X}$ of gambles to a non-empty subset of this set:
  \begin{equation*}
    \emptyset \neq \opt(\mathcal{X}|A) \subseteq \mathcal{X}.
  \end{equation*}
\end{definition}

Note that common uses of choice functions in social choice theory, such as by Sen~\cite[p.~63, ll.~19--21]{1977:sen} do not consider conditioning, and define choice functions for arbitrary sets of options (not for gambles only).

\subsection{Normal Form Operator Induced by a Choice Function}

Now,
given a choice function $\opt$, we naturally arrive at a normal form operator $\normoper_{\opt}$, simply by applying $\opt$ on the set of all gambles associated with the tree $\tree$ and then finding the corresponding set of normal form decisions.

\begin{definition}\label{def:normoper}
  Given any choice function $\opt$, and any decision tree $\tree$ with $\treeevent{\tree}\neq\emptyset$, we define
\begin{equation*}
  \normoper_{\opt}(\tree)=\{\atree \in \nfd(\tree)\colon \normgambles(\atree)\subseteq \opt(\normgambles(\tree)|\treeevent{\tree})\}.
\end{equation*}
\end{definition}
Of course, since the $\atree$ are normal form decisions, $\normgambles(\atree)$ is always a singleton in this definition. In particular, the following important equality holds,
\begin{equation}\label{eq:gambofnormoptisopt}
  \normgambles(\normoper_{\opt}(\tree))=\opt(\normgambles(\tree)|\treeevent{\tree}).
\end{equation}

It follows immediately that 
$\normoper_{\opt}$ respects strategic equivalence:

\begin{theorem}\label{thm:normopt:strategically:equivalent}
  If $\tree_1$ and $\tree_2$ are strategically equivalent and $\treeevent{\tree_1}=\treeevent{\tree_2}\neq\emptyset$,
  then $
  \normgambles(\normoper_{\opt}(\tree_1)) = \normgambles(\normoper_{\opt}(\tree_2))
  $.
\end{theorem}

This is of course an attractive property, as 
there are always many strategically equivalent trees representing the same problem: Theorem~\ref{thm:normopt:strategically:equivalent} guarantees that all equivalent representations yield the same solution.

When studying subtree perfectness, we consider $\normoper_{\opt}$ for arbitrary subtrees. To ensure that $\normoper_{\opt}$ can be applied on each of these, we need:
\begin{definition}\label{def:consistent:decision:trees}
  A decision tree $\tree$ is called \emph{consistent} if for every node $N$ of $\tree$,
  $$\treeevent{\subtreeat{\tree}{N}}\neq\emptyset.$$
\end{definition}
Clearly, if a decision tree $\tree$ is consistent, then for any node $N$ in $\tree$, $\subtreeat{\tree}{N}$ is also consistent. We study only consistent decision trees because we consider $\normoper_{\opt}(\subtreeat{\tree}{N})$ for any node $N$ in $\tree$, which is impossible when $\treeevent{\subtreeat{\tree}{N}}=\emptyset$.

Usually, one does not consider events which conflict with preceding events, hence consistency is satisfied. However, due to an oversight, some branch of a chance node might represent an event that cannot occur: such tree can always be made consistent by removing those nodes whose conditioning event is empty.

Not all sets of gambles can be represented by a consistent decision tree:
\begin{definition}\label{def:gambles:consistent}
  Let $A$ be any non-empty event, and let $\mathcal{X}$ be a non-empty finite set of gambles. Then the following conditions are equivalent; if any (hence all) of them are satisfied, we say that $\mathcal{X}$ is \emph{$A$-consistent}.
  \begin{enumerate}[label=(\Alph*)]
  \item\label{def:gambles:consistent:via:trees} There is a consistent decision tree $T$ with $\treeevent{\tree}=A$ and $\normgambles(\tree)=\mathcal{X}$.
  \item\label{def:gambles:consistent:via:inverse:map} For every $r\in\rewardset$ and every $X\in\mathcal{X}$ such that $X^{-1}(r)\neq\emptyset$, it holds that $X^{-1}(r)\cap A\neq\emptyset$.
  \end{enumerate}
  We will also say that a gamble $X$ is $A$-consistent whenever $\{X\}$ is $A$-consistent.
\end{definition}


\section{Subtree Perfectness}
\label{sec:counterfactuals}

We now define subtree perfectness, for both types of operator, and find necessary and sufficient conditions on $\opt$ for $\normoper_{\opt}$ to be subtree perfect. 

\subsection{Example and Definition}

First, we illustrate subtree perfectness by an example. Suppose we apply an extensive form operator to the tree $\tree$ in Fig.~\ref{fig:lake:tree:basic}. This operator will delete some (possibly none) of the decision arcs at $N=\decnode[1](1)[1]$. If the operator would delete the same arcs at $N$ regardless of the larger tree in which $\subtreeat{\tree}{N}$ is embedded, then the operator is subtree perfect. If the operator does not have this property (for instance, if the solution of $\tree$ after $N$ were to depend on consequences of $d_{\compl{S}}$ or $S_2$), then it fails subtree perfectness. Hammond~\cite[p.~34]{1988:hammond} calls subtree perfectness for behaviour norms (extensive form operators) \emph{consistency}.

The definition for subtree perfectness for a normal form operator requires the following extension to Definition~\ref{def:subtreeat}.
\begin{definition}
  If $\setoftrees$ is a set of decision trees and $N$ a node, then
  \begin{equation*}
    \subtreeat{\setoftrees}{N}=\{\subtreeat{\tree}{N}\colon\tree\in \setoftrees\text{ and }N\text{ in }\tree\}.
  \end{equation*}
\end{definition}

\begin{definition}
  An extensive form operator $\extoper$ is called \emph{subtree perfect} if for every consistent decision tree $\tree$ and every node $N$ such that $N$ is in $\extoper(\tree)$,
  \begin{equation*}
    \subtreeat{\extoper(\tree)}{N}=\extoper(\subtreeat{\tree}{N}).
  \end{equation*}
  A normal form operator $\normoper$ is called \emph{subtree perfect} if for every consistent decision tree $\tree$ and every node $N$ which is in at least one element of $\normoper(\tree)$,
  \begin{equation*}
    \subtreeat{\normoper(\tree)}{N}=\normoper(\subtreeat{\tree}{N}).
  \end{equation*}
\end{definition}

In other words, for a subtree perfect operator, it does not matter whether we first restrict to a subtree and then optimize, or first optimize and only then restrict to the subtree: subtree perfectness means that optimization and restriction commute, as in Fig.~\ref{fig:factual:commutativity}. 

\begin{figure}
 \begin{center}
  \begin{tikzpicture}
   \node (topleft) {$\tree$} [level distance = 8em] 
	child[parent anchor = east, child anchor = west, grow'=east]{
		node (topright) {$\subtreeat{\tree}{N}$}
		child[parent anchor = south, child anchor = north, grow'=south, level distance = 4em]
		{
			node (bottomright) {$\extoper(\subtreeat{\tree}{N})$}
			edge from parent[->]
			node[right]{optimise}
		}
		edge from parent[->]
		node[above]{restrict}
	}	
	child[parent anchor = south, child anchor = north, grow'=south]{
		node (bottomleft) {$\extoper(\tree)$}
		child[parent anchor = east, child anchor = west, grow'=east]{
			node(reallybottomright){$\subtreeat{\extoper(\tree)}{N}$}
			edge from parent[->]
			node[below]{restrict}
		}
		edge from parent[->]
		node[left]{optimise}
	};
\draw ([xshift=-2pt]reallybottomright.north)-- ([xshift=-2pt]bottomright.south);
\draw (reallybottomright) -- node[right]{if $N$ in $\extoper(\tree)$} (bottomright);

  \end{tikzpicture}
  \caption{For a subtree perfect extensive form operator, optimization and restriction commute.}
\label{fig:factual:commutativity}
 \end{center}
\end{figure}
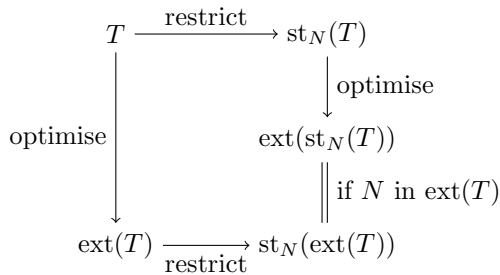

The extensive form operator $\extoper_P$ corresponding to the usual backward induction using expected utility is well known to be subtree perfect, provided probabilities at chance nodes are non-zero. Also, the usual normal form operator $\normoper_P$ corresponding to maximizing expected utility over all normal form decisions is subtree perfect, because $\extoper_P$ is equivalent to $\normoper_P$.

Not all choice functions are subtree perfect:

\begin{example}\label{ex:maximality:counterfactual}
  Let $\tree$ be the decision tree in Fig.~\ref{fig:maximality:counterfactual:tree}, where $X$, $Y$, and $Z$ are its normal form gambles. Under point-wise dominance, $X$ and $Y$ are incomparable, as are $Y$ and $Z$. Hence, $\normoper(\subtreeat{\tree}{N})$ is $\{X,Y\}$ (where we conveniently identified normal form decisions with their normal form gambles). But $\normoper(\tree) = \opt(\{X,Y,Z\})=\{Y,Z\}$ as clearly $Z$ dominates $X$. Restricting this solution to $\subtreeat{\tree}{N}$ gives the normal form solution $\{Y\}$. Concluding,
\begin{equation*}
  \{X,Y\}=\normoper(\subtreeat{\tree}{N})\neq\subtreeat{\normoper(\tree)}{N}=\{Y\}
\end{equation*}
and therefore the normal form operator induced by $\opt$ lacks subtree perfectness.
\end{example}

\begin{figure}
  \begin{center}
    \begin{tikzpicture}
      [minimum size=2em,parent anchor=east,child anchor=west,grow'=east,transform shape]
	\node at (-3,0){
    \small
    \begin{tabular}{c|cc}
      & $A$ & $\compl{A}$ \\
      \hline
      $X$ & $-1$ & $-1$ \\
      $Y$ & $-2$ & $2$ \\
      $Z$ & $0$ & $0$
   \end{tabular}
    };
	\node[draw,rectangle]{}
	[sibling distance = 2em, level distance = 4em]
		child{
			node[draw,rectangle]{$N$}
				[sibling distance = 2em]
			child{
				node[right]{$-1$}
			}
			child{
				node[draw,circle]{}
				child{
					node[right]{$-2$}
                    edge from parent
                    node[above,sloped]{$A$}
				}
				child{
					node[right]{$2$}
                    edge from parent
                    node[below,sloped]{$\compl{A}$}
				}
			}
		}
		child{
			node[right]{$0$}
		};
\end{tikzpicture}
\caption{Decision tree for Example~\ref{ex:maximality:counterfactual}.}
\label{fig:maximality:counterfactual:tree}
\end{center}
\end{figure}
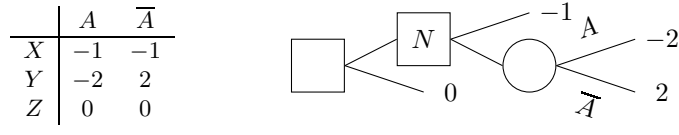

\subsection{Necessary and Sufficient Conditions}

In this section, we work extensively with normal form solutions, which are sets of trees. Therefore, we extend $\normgambles$, $\chancenodemixture$, and $\decnodeunion$, to sets of trees:
\begin{definition}
For any set of decision trees $\setoftrees$,
$
\normgambles(\setoftrees)=\bigcup_{\tree\in\setoftrees}\normgambles(\tree).
$
\end{definition}

\begin{definition}
For any sets of consistent decision trees $\setoftrees_1$, \dots, $\setoftrees_n$, and any partition $E_1$, \dots, $E_n$,
let
\begin{align*}
\textstyle
\bigchancenodemixture_{i=1}^n E_i \setoftrees_i
&=
\textstyle
\left\{\bigchancenodemixture_{i=1}^n E_i \tree_i\colon \tree_i \in \setoftrees_i\right\}, \\
\textstyle
\bigdecnodeunion_{i=1}^n \setoftrees_i
&=
\textstyle
\left\{\bigdecnodeunion_{i=1}^n \tree_i\colon\tree_i\in\setoftrees_i\right\}.
\end{align*}
\end{definition}

For sets of trees, the $\normgambles$ operator keeps working as expected:
\begin{align}\label{eq:normgambles:chancenodes:for:sets:of:trees}
\textstyle
  \normgambles\left(\bigchancenodemixture_{i=1}^n E_i \setoftrees_i\right)
  &=
\textstyle
  \biggambplus_{i=1}^n E_i \normgambles(\setoftrees_i),
\\
\label{eq:normgambles:decisionnodes:for:sets:of:trees}
\textstyle
 \normgambles\left(\bigdecnodeunion_{i=1}^n \setoftrees_i\right)
 &=
\textstyle
 \bigcup_{i=1}^n \normgambles(\setoftrees_i).
\end{align}
One should also observe that
$
  \normgambles(\tree)=\normgambles(\nfd(\tree)).
$

The next three properties turn out to be necessary and sufficient for subtree perfectness of normal form operators induced by a choice function.

\begin{property}[Conditioning Property]\label{prop:conditioning:property}
Let $A$ be a non-empty event, and let $\mathcal{X}$ be a non-empty finite $A$-consistent set of gambles, with $\{X,Y\}\subseteq\mathcal{X}$ such that $AX=AY$. If $X\in\opt(\mathcal{X}|A)$, then $Y\in\opt(\mathcal{X}|A)$.
\end{property}

This property is not found in the accounts of Hammond and McClennen, because their decision trees are different from ours. If their methods were adapted to our decision trees, then the conditioning property would appear in their work too. The property is more of a technical detail than an important point, since it states that, if two gambles are equal on an event, one gamble cannot be preferred to the other given that event. Such a property would be common and desirable for a choice function, regardless of its implications for subtree perfectness.

\begin{property}[Intersection property]\label{prop:intersection:property}
  For any event $A\neq\emptyset$ and non-empty finite $A$-consistent sets of gambles $\mathcal{X}$ and $\mathcal{Y}$ such that $\mathcal{Y}\subseteq\mathcal{X}$ and $\opt(\mathcal{X}|A)\cap\mathcal{Y}\neq\emptyset$,
  \begin{equation*}
    \opt(\mathcal{Y}|A) = \opt(\mathcal{X}|A)\cap\mathcal{Y}.
  \end{equation*}
\end{property}

The intersection property is precisely Arrow's (C4) \cite[p.~123]{1959:arrow}.

For the next property, we use the following extension of the $\gambplus$ notation: if $A$ is a non-trivial event (non-empty and not $\pspace$), then
\begin{equation*}
 A\mathcal{X} \gambplus \compl{A}Z = \{AX \gambplus \compl{A}Z \colon X\in\mathcal{X}\}.
\end{equation*}

\begin{property}[Mixture property]\label{prop:mixture:property}
  For any events $A$ and $B$ such that $A\cap B \neq \emptyset$ and $\compl{A}\cap B\neq \emptyset$, any $\compl{A}\cap B$-consistent gamble $Z$, and any non-empty finite $A\cap B$-consistent set of gambles $\mathcal{X}$,
  \begin{equation*}
    \opt(A\mathcal{X}\gambplus\compl{A}Z|B)=A\opt(\mathcal{X}|A\cap B)\gambplus\compl{A}Z.
  \end{equation*}
\end{property}

This property is a form of the well known independence principle (see for example \cite[p.~44]{1990:mcclennen}). It has strongest similarities to McClennen's independence for choice \cite[p.~57]{1990:mcclennen}, and Arrow's conditional preference \cite[p.~257]{1966:arrow}.

Property~\ref{prop:intersection:property} has a many equivalent formulations. The following three give interesting alternative interpretations, and are useful in some of the proofs.

\begin{property}[Strong path independence]\label{prop:strong:path:independence}
  For any non-empty event $A$ and any non-empty finite $A$-consistent sets of gambles $\mathcal{X}_1$, \dots ,$\mathcal{X}_n$, there is a non-empty $\mathcal{I}\subseteq\{1,\dots,n\}$ such that
  \begin{equation*}
    \opt\Bigg(\bigcup_{i=1}^n \mathcal{X}_i\Bigg|A\Bigg) = \bigcup_{i\in\mathcal{I}}\opt(\mathcal{X}_i|A)
  \end{equation*}
\end{property}

\begin{property}[Very strong path independence]\label{prop:very:strong:path:independence}
  For any non-empty event $A$ and any non-empty finite $A$-consistent sets of gambles $\mathcal{X}_1$, \dots ,$\mathcal{X}_n$,
  \begin{equation*}
    \opt\Bigg(\bigcup_{i=1}^n \mathcal{X}_i\Bigg|A\Bigg)
    =\bigcup_{
      \substack{
        i=1\\
        \mathcal{X}_i\cap\opt(\cup_{i=1}^n \mathcal{X}_i|A)\neq\emptyset
      }
    }^n
    \opt(\mathcal{X}_i|A)
  \end{equation*}
\end{property}

\begin{property}[Total preorder]\label{prop:total:order:property}
  For every event $A\neq\emptyset$, there is a total preorder $\succeq_A$ on $A$-consistent gambles such that for every non-empty finite set of $A$-consistent gambles $\mathcal{X}$,
  \begin{equation*}
    \opt(\mathcal{X}|A)=\{X\in\mathcal{X}\colon(\forall Y\in\mathcal{X})(X\succeq_A Y)\}
  \end{equation*}
\end{property}

The total preorder property essentially boils down to Arrow's (C5), which is also called the \emph{weak axiom of revealed preference} \cite[p.~123]{1959:arrow}. More equivalents are given by Arrow \cite[(C1)]{1959:arrow}, Houthakker \cite[p.~163]{1950:houthakker}, and Ville \cite[p.~123]{1951:ville}.

\begin{lemma}\label{lemma:strong:path:independence:equivalence}
  Properties~\ref{prop:intersection:property}, \ref{prop:strong:path:independence}, \ref{prop:very:strong:path:independence} and~\ref{prop:total:order:property} are equivalent.
\end{lemma}
\begin{proof}
  Property~\ref{prop:total:order:property}$\implies$Property~\ref{prop:very:strong:path:independence}$\implies$Property~\ref{prop:strong:path:independence}$\implies$Property~\ref{prop:intersection:property}. Immediate.
  Property~\ref{prop:intersection:property}$\implies$Property~\ref{prop:total:order:property}. See Arrow's (C4)$\iff$(C5) \cite[p.~124, Thm.~1]{1959:arrow}.
\end{proof}

To show that Properties~\ref{prop:conditioning:property}, \ref{prop:intersection:property} and~\ref{prop:mixture:property} are necessary and sufficient for subtree perfectness of $\normoper_{\opt}$, we require several lemmas. The proofs are long but mostly tedious and straightforward, and so are omitted.

For a decision tree $\tree$, $\children(\tree)$ is the set of child nodes of the root node of $\tree$. 

\begin{lemma}\label{lemma:factuality:at:children}
  Let $\normoper$ be any normal form operator. Let $\tree$ be a consistent decision tree. If,
  \begin{enumerate}[label=(\roman*)]
  \item\label{lemma:factuality:at:children:childcondition} for all nodes $K\in\children(\tree)$ such that $K$ is in at least one element of $\normoper(\tree)$,
    \begin{equation*}
      \subtreeat{\normoper(\tree)}{K}=\normoper(\subtreeat{\tree}{K}),
    \end{equation*}
  \item\label{lemma:factuality:at:children:grandchildcondition} and, for all nodes $K\in\children(\tree)$, and all nodes $L\in\subtreeat{\tree}{K}$ such that $L$ is in at least one element of $\normoper(\subtreeat{\tree}{K})$,
  \begin{equation*}
    \subtreeat{\normoper(\subtreeat{\tree}{K})}{L}=\normoper(\subtreeat{\subtreeat{\tree}{K}}{L}),
  \end{equation*}
  \end{enumerate}
  then, for all nodes $N$ in $\tree$ such that $N$ is in at least one element of $\normoper(\tree)$,
  \begin{equation*}
    \subtreeat{\normoper(\tree)}{N}=\normoper(\subtreeat{\tree}{N}).
  \end{equation*}
\end{lemma}

\begin{lemma}\label{lemma:opt:of:setsums:strong:equality}
  Let $A_1$, \dots, $A_n$ be a finite partition of $\pspace$, and let $B$ be an event such that $A_i\cap B\ne\emptyset$ for all $i$. Let $\mathcal{X}_1$, \dots, $\mathcal{X}_n$ be a finite family of non-empty finite sets of gambles, where $\mathcal{X}_i$ is $A_i\cap B$-consistent. If a choice function $\opt$ satisfies Properties~\ref{prop:intersection:property} and~\ref{prop:mixture:property}, then
  \begin{equation}\label{eq:opt:of:setsums:strong:equality}
    \opt\Bigg(\biggambplus_{i=1}^n A_i \mathcal{X}_i\Bigg|B\Bigg) = \biggambplus_{i=1}^n A_i \opt(\mathcal{X}_i | A_i\cap B).
  \end{equation}
\end{lemma}

\begin{lemma}\label{lemma:I:Istar:equality}
  Consider a consistent decision tree $\tree$ whose root is a decision node, so $\tree=\bigdecnodeunion_{i=1}^n \tree_i$, and any choice function $\opt$. For each tree $\tree_i$, let $N_i$ be its root. Then, $N_i$ is in at least one element of $\normoper_{\opt}(\tree)$ if and only if
  \begin{equation}\label{eq:lemma:I:Istar:equality}
    \normgambles(\tree_i) \cap \opt(\normgambles(\tree)|\treeevent{T})\neq\emptyset.
  \end{equation}
\end{lemma}

\begin{lemma}\label{lemma:gamble:normopt:equivalence:chance:nodes}
  For any consistent decision tree $\tree=\bigchancenodemixture_{i=1}^n E_i\tree_i$, and any choice function $\opt$ satisfying Property~\ref{prop:conditioning:property},
  \begin{equation}\label{lemma:gamble:normopt:equivalence:chance:nodes:helper}
    \normgambles(\normoper_{\opt}(\tree)) = \biggambplus_{i=1}^n E_i \normgambles(\normoper_{\opt}(\tree_i))
  \end{equation}
  implies
  \begin{equation*}
    \normoper_{\opt}(\tree) = \bigchancenodemixture_{i=1}^n E_i \normoper_{\opt}(\tree_i).
  \end{equation*}
\end{lemma}

\begin{lemma}\label{lemma:gamble:normopt:equivalence:decision:nodes}
  For any consistent decision tree $\tree=\bigdecnodeunion_{i=1}^n\tree_i$ and any choice function $\opt$ satisfying Property~\ref{prop:intersection:property},
  \begin{equation}\label{lemma:gamble:normopt:equivalence:decision:nodes:helper}
    \normgambles(\normoper_{\opt}(\tree)) = \bigcup_{i\in\mathcal{I}} \normgambles(\normoper_{\opt}(\tree_i))
  \end{equation}
  implies
  \begin{equation*}
    \normoper_{\opt}(\tree) = \nfd\Bigg(\bigdecnodeunion_{i\in\mathcal{I}} \normoper_{\opt}(\tree_i)\Bigg),
  \end{equation*}
  where $\mathcal{I}=\{i\in\{1,\dots,n\}\colon\normgambles(\tree_i) \cap \opt(\normgambles(\tree)|\treeevent{T})\neq\emptyset\}$.
\end{lemma}

The next lemma shows necessity of Properties~\ref{prop:conditioning:property}, \ref{prop:intersection:property}, and~\ref{prop:mixture:property} for subtree perfectness. Interestingly, the proof only involves the two decision trees in Figure~\ref{fig:conditioning:intersection:mixture:are:necessary:for:factuality}.

\begin{lemma}\label{lemma:conditioning:intersection:mixture:are:necessary:for:factuality}
  If $\normoper_{\opt}$ is subtree perfect, then $\opt$ satisfies Properties~\ref{prop:conditioning:property}, \ref{prop:intersection:property}, and~\ref{prop:mixture:property}.
\end{lemma}

  \begin{figure}
    \begin{center}
      \begin{tikzpicture}
        [minimum size=2em,parent anchor=east,child anchor=west,grow'=east,transform shape]
        \node[draw,circle]{}
        [level distance = 4em,sibling distance=2em]
          child{
            node[draw,rectangle]{$N$}
            [sibling distance = 1em]
            child{
              node[right]{$X_1$}
            }
            child{
              node[right]{\tiny$\vdots$}
            }
            child{
              node[right]{$X_n$}
            }
            edge from parent
            node[above,sloped]{$A$}
          }
          child{
            node[right]{$Z$}
            edge from parent
            node[below,sloped]{$\compl{A}$}
          };        
      \end{tikzpicture}
      \hspace{2em}
    \begin{tikzpicture}
      [minimum size=2em,parent anchor=east,child anchor=west,grow'=east,transform shape]
      \node[draw,rectangle]{}
      [level distance = 4em, sibling distance=4em]
        child{
          node[draw,rectangle]{$N$}
          [sibling distance = 1em]
          child{
            node[right]{$Y_1$}
          }
          child{
            node[right]{\tiny$\vdots$}
          }
          child{
            node[right]{$Y_m$}
          }
        }
        child{
          node[draw,rectangle]{}
          [sibling distance = 1em]
          child{
            node[right]{$X_1$}
          }
          child{
            node[right]{\tiny$\vdots$}
          }
          child{
            node[right]{$X_n$}
          }
        };        
    \end{tikzpicture}
    \caption{Decision trees for Lemma~\ref{lemma:conditioning:intersection:mixture:are:necessary:for:factuality}.}
    \label{fig:conditioning:intersection:mixture:are:necessary:for:factuality}
  \end{center}
\end{figure}
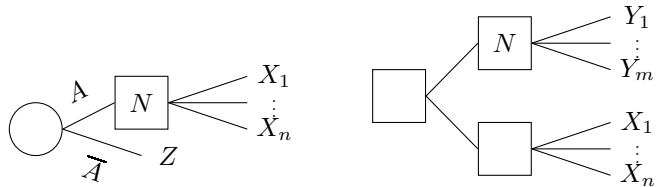

We are now ready to establish that Properties~\ref{prop:conditioning:property}, \ref{prop:intersection:property}, and~\ref{prop:mixture:property} are also sufficient for subtree perfectness. 

\begin{theorem}[Subtree perfectness theorem]\label{thm:conditions:to:be:factual}
  A normal form operator $\normoper_{\opt}$ is subtree perfect if and only if $\opt$ satisfies Properties~\ref{prop:conditioning:property}, \ref{prop:intersection:property} and~\ref{prop:mixture:property}.
\end{theorem}
\begin{proof}
  ``only if''. See Lemma~\ref{lemma:conditioning:intersection:mixture:are:necessary:for:factuality}.
  
  ``if''.
  We proceed by structural induction on all possible arguments of $\normoper_{\opt}$, that is, on all consistent decision trees. In the base step, we prove the implication for trees consisting of only a single node. In the induction step, we prove that if the implication holds for the subtrees at every child of the root node, then the implication also holds for the whole tree.

  First, if the decision tree $\tree$ has only a single node (a reward, and no further children) then 
subtree perfectness is trivially satisfied.

  Next, suppose that the consistent decision tree $\tree$ has multiple nodes. Let $\{N_1,\dots,N_n\}=\children(\tree)$ be the children of the root node of $T$, and let $\tree_i=\subtreeat{\tree}{N_i}$. The induction hypothesis is: subtree perfectness holds for all subtrees at every child of the root node, that is, for all $\tree_i$. More precisely, for all $i\in\{1,\dots,n\}$, and all nodes $L\in\tree_i$ which are in at least one element of $\normoper_{\opt}(\tree_i)$
  \begin{equation*}
    \subtreeat{\normoper_{\opt}(\tree_i)}{L}=\normoper_{\opt}(\subtreeat{\tree_i}{L}).
  \end{equation*}
  We must show that
  \begin{equation*}
    \subtreeat{\normoper_{\opt}(\tree)}{N}=\normoper_{\opt}(\subtreeat{\tree}{N})
  \end{equation*}
  for all nodes $N$ in $\tree$ such that $N$ is in at least one element of $\normoper_{\opt}(\tree)$. By Lemma~\ref{lemma:factuality:at:children}, and the induction hypothesis, it suffices to prove the above equality only for $N\in\children(\tree)$, that is, it suffices to show that
  \begin{equation}\label{eq:conditions:to:be:factual:help:1}
    \subtreeat{\normoper_{\opt}(\tree)}{N_i}=\normoper_{\opt}(\tree_i)
  \end{equation}
  for each $i\in\{1,\dots,n\}$ such that $N_i$ is in at least one element of $\normoper_{\opt}(\tree)$. 

  If $\tree$ has a chance node as its root, that is, $\tree=\bigchancenodemixture_{i=1}^nE_i\tree_i$, then all $N_i$ are actually in every element of $\normoper_{\opt}(\tree)$, so we must simply establish Eq.~\eqref{eq:conditions:to:be:factual:help:1} for all $i\in\{1,\dots,n\}$. Observe that, if we can establish
  \begin{equation}\label{eq:conditions:to:be:factual:help:2}
    \normoper_{\opt}(\tree)=\bigchancenodemixture_{i=1}^n\event[i]\normoper_{\opt}(\tree_i),
  \end{equation}
  then Eq.~\eqref{eq:conditions:to:be:factual:help:1} follows immediately.
  Indeed, by Eq.~\eqref{eq:gambofnormoptisopt},
  \begin{align*}
    \normgambles(\normoper_{\opt}(\tree))
    &=
    \opt(\normgambles(\tree)|\treeevent{T})
    \\
    \intertext{and by the definition of the $\normgambles$ operator, Eq.~\eqref{eq:normgambles:chancenodes} in particular,}
    &=
    \opt\Bigg(\biggambplus_{i=1}^n\event[i]\normgambles(\tree_i)\Bigg|\treeevent{T}\Bigg)
    \\
    \intertext{and so by Lemma~\ref{lemma:opt:of:setsums:strong:equality},}
    &=
    \biggambplus_{i=1}^n\event[i]\opt(\normgambles(\tree_i)|\treeevent{T}\cap \event[i])
    \\
    \intertext{so, since $\treeevent{T}\cap \event[i]=\treeevent{T_i}$, and again by Eq.~\eqref{eq:gambofnormoptisopt},}
    &=
    \biggambplus_{i=1}^n\event[i]\normgambles(\normoper_{\opt}(\tree_i))
  \end{align*}
  Whence, Eq.~\eqref{eq:conditions:to:be:factual:help:2} follows by Lemma~\ref{lemma:gamble:normopt:equivalence:chance:nodes}.
  
  Finally, assume that $\tree$ has a decision node as its root, that is, $\tree=\bigdecnodeunion_{i=1}^n \tree_i$. Let $\mathcal{I}$ be the subset of $\{1,\dots,n\}$ such that $i\in\mathcal{I}$ if and only if $N_i$ is in at least one element of $\normoper_{\opt}(\tree)$. We must establish Eq.~\eqref{eq:conditions:to:be:factual:help:1} for all $i\in\mathcal{I}$. Equivalently, we must show that
  \begin{equation}\label{eq:conditions:to:be:factual:help:3}
    \normoper_{\opt}(\tree)=\nfd\Bigg(\bigdecnodeunion_{i\in\mathcal{I}}\normoper_{\opt}(\tree_i)\Bigg).
  \end{equation}
  Indeed, by Eq.~\eqref{eq:gambofnormoptisopt},
  \begin{align*}
    \normgambles(\normoper_{\opt}(\tree))
    &=
    \opt(\normgambles(\tree)|\treeevent{T})
    \\
    \intertext{and by the definition of the $\normgambles$ operator, Eq.~\eqref{eq:normgambles:decisionnodes} in particular,}
    &=
    \opt\Bigg(\bigcup_{i=1}^n\normgambles(\tree_i)\Bigg|\treeevent{T}\Bigg)
    \\
    \intertext{and so by Property~\ref{prop:very:strong:path:independence},}
    &=
    \bigcup_{i\in\mathcal{I}^*}\opt(\normgambles(\tree_i)|\treeevent{T}),
    \\
    \intertext{where $\mathcal{I}^*=\{i\in\{1,\dots,n\}\colon\normgambles(\tree_i) \cap \opt(\normgambles(\tree)|\treeevent{T})\neq\emptyset\}$, and so because $\treeevent{T}=\treeevent{T_i}$, and again by Eq.~\eqref{eq:gambofnormoptisopt},}
    &=
    \bigcup_{i\in\mathcal{I}^*} \normgambles(\normoper_{\opt}(\tree_i)).
  \end{align*}
  Hence, the conditions of Lemma~\ref{lemma:gamble:normopt:equivalence:decision:nodes} are satisfied, and $\mathcal{I}^*=\mathcal{I}$ by Lemma~\ref{lemma:I:Istar:equality}, so Eq.~\eqref{eq:conditions:to:be:factual:help:3} is established.
\end{proof}

\section{Backward Induction}
\label{sec:backward:induction:and:factuality}


Although Selten's \cite{1975:selten} definition of subgame perfectness does not explicitly refer to backward induction, one of its aims is unmistakably to identify when a game's equilibrium point can be found by backward induction. In fact, when maximizing expected utility, as in Selten's work, there seems little reason to distinguish between subtree perfectness and backward induction because the optimal decisions are essentially unique: multiple optimal decisions will have the same expectation, so it suffices to consider only a single representative.

However, if our concept of choice has no reference to expectation, multiple optimal decisions are not necessarily equivalent in any sense. Consequently, backward induction does not need to be tied to subtree perfectness.

Therefore, elsewhere \cite{2008:huntley:troffaes::impdectrees:smps}, we suggested the following backward induction method, which generalizes classical backward induction to arbitrary choice functions, and which is useful when $\normgambles(\tree)$ is very large and applying $\opt$ in one go is not feasible. We first extend $\normoper_{\opt}$ to act on sets of decision trees.

\begin{definition}
  Given a choice function $\opt$ and any set $\setoftrees$ of consistent decision trees, where $\treeevent{\tree}=A$ for all $\tree\in\setoftrees$,
  \begin{equation*}
    \normoper_{\opt}(\setoftrees) = \{\atree\in\nfd(\setoftrees)\colon \normgambles(\atree)\subseteq \opt(\normgambles(\setoftrees)|A)\}.
  \end{equation*}
\end{definition}

The goal of our backward induction algorithm is to reach a normal form solution of $\tree$ by finding normal form solutions of subtrees of $\tree$, and using these solutions to remove some elements of $\normgambles(\tree)$ before applying $\opt$:

\begin{subequations}
\begin{definition}
  \label{def:backopt}
  The normal form operator $\backopt$ is defined for any consistent decision tree $\tree$ through:
  \begin{itemize}
  \item If $\tree$ consists of only a leaf with reward $\reward\in\rewardset$, then
    \begin{equation}\backopt(\tree) = \{T\}.\end{equation}
  \item If $\tree$ has a chance node as root, that is, $\tree=\bigchancenodemixture_{i=1}^n E_i \tree_i$, then
    \begin{equation}\label{eq:backopt:chance}
      \backopt\left(\bigchancenodemixture_{i=1}^n E_i \tree_i\right)=\normoper_{\opt}\left(
        \bigchancenodemixture_{i=1}^n
        E_i
        \backopt\left(\tree_i\right)
      \right)
    \end{equation} 
  \item If $\tree$ has a decision node as root, that is, if $\tree=\bigdecnodeunion_{i=1}^n \tree_i$, then
    \begin{equation}\label{eq:backopt:decision}
      \backopt\left(\bigdecnodeunion_{i=1}^n\tree_i\right)
      =
      \normoper_{\opt}\left(\bigdecnodeunion_{i=1}^n
        \backopt(\tree_i)\right).
    \end{equation}
  \end{itemize}
\end{definition}
\end{subequations}

This algorithm is almost identical to that suggested by Kikuti et al.~\cite{2005:kikuti}: they apply $\opt$ only at decision nodes, whereas we apply $\opt$ at all type of node.

It is instructive to compare the definition of the $\normgambles$ operator (Definition~\ref{def:normgambles}) with Definition~\ref{def:backopt}. The main difference is that $\backopt$ inserts $\normoper_{\opt}$ at every stage, to remove as many normal form decisions as possible, early on.

If $\backopt=\normoper_{\opt}$, we can use the former as an efficient way of calculating the latter. Of course, this only works if a normal form gamble that is non-optimal in a subtree at a node cannot be part of an optimal gamble in the full tree. It is well known that choice functions exist for which this property does not hold: for examples, see LaValle and Wapman \cite{1986:lavalle}, Jaffray \cite{1999:jaffray}, and Seidenfeld \cite{2004:seidenfeld}. In such cases, there exist trees such that $\backopt(\tree) \neq \normoper_{\opt}(\tree)$.

The following four properties are necessary and sufficient for $\backopt$ to coincide with $\normoper_{\opt}$.

\begin{property}[Backward conditioning property]\label{prop:backward:conditioning:property}
  Let $A$ and $B$ be events such that $A\cap B\neq\emptyset$ and $\compl{A}\cap B \neq \emptyset$, and let $\mathcal{X}$ be a non-empty finite $A\cap B$-consistent set of gambles, with $\{X,Y\}\subseteq \mathcal{X}$ such that $AX=AY$. Then $X\in\opt(\mathcal{X}|A\cap B)$ implies $Y\in\opt(\mathcal{X}|A\cap B)$ whenever there is a non-empty finite $\compl{A}\cap B$-consistent set of gambles $\mathcal{Z}$ such that, for at least one $Z\in\mathcal{Z}$,
  \begin{equation*}
    AX \gambplus \compl{A}Z \in\opt(A\mathcal{X} \gambplus \compl{A}\mathcal{Z} | B).
  \end{equation*}
\end{property}

This is clearly just a minor relaxation of Property~\ref{prop:conditioning:property}.


\begin{property}[Insensitivity of optimality to the omission of non-optimal elements]\label{prop:insensitivity:to:non:optimal:elements}
For any event $A\neq\emptyset$, and any non-empty finite $A$-consistent sets of gambles $\mathcal{X}$ and $\mathcal{Y}$,
\begin{equation*}
\opt(\mathcal{X}|A) \subseteq \mathcal{Y} \subseteq \mathcal{X} \Rightarrow \opt(\mathcal{Y}|A) = \opt(\mathcal{X}|A).
\end{equation*}
\end{property}

If $\opt$ satisfies this property, then removing non-optimal elements from a set does not affect whether or not each of the remaining elements is optimal. The property is called `insensitivity to the omission of non-optimal elements' by De~Cooman and Troffaes \cite{2005:decooman}, and `property $\epsilon$' by Sen \cite{1977:sen} who attributes this designation to Douglas Blair.

\begin{property}[Preservation of non-optimality under the addition of elements]\label{prop:preservation:under:addition:of:elements}
For any event $A\neq\emptyset$, and any non-empty finite $A$-consistent sets of gambles $\mathcal{X}$ and $\mathcal{Y}$,
\begin{equation*}
\mathcal{Y} \subseteq \mathcal{X} \Rightarrow \opt(\mathcal{Y}|A) \supseteq \opt(\mathcal{X}|A) \cap \mathcal{Y}.
\end{equation*}
\end{property}

This is `property $\alpha$' in Sen \cite{1977:sen}, Axiom 7 in Luce and Raiffa \cite[p.~288]{1957:luce}, and `independence of irrelevant alternatives' in Radner and Marschak \cite{1954:radnermarschak}.

\begin{property}[Backward mixture property]\label{prop:backward:mixture:property}
For any events $A$ and $B$ such that $B\cap A \neq \emptyset$ and $B\cap\compl{A}\neq\emptyset$, any $B\cap\compl{A}$-consistent gamble $Z$, and any non-empty finite $B\cap A$-consistent set of gambles $\mathcal{X}$,
\begin{equation*}
\opt\left(A\mathcal{X} \gambplus \compl{A}Z|B\right) \subseteq A\opt(\mathcal{X}|A\cap B) \gambplus \compl{A}Z.
\end{equation*}
\end{property}

This is an ``inclusion-only'' version of Property~\ref{prop:mixture:property}. We do not need the full property because backward induction moves from right to left.


The proof of the following theorem is, up to obvious details, identical to the proof of Theorem~\ref{thm:conditions:to:be:factual}, and is therefore left as an amusing exercise to the reader.

\begin{theorem}[Backward induction theorem]\label{thm:backopt:normopt:equivalence}
  A normal form operator $\normoper_{\opt}$ satisfies backward induction (that is, $\backopt(\tree)=\normoper_{\opt}(\tree)$ for every consistent decision tree $\tree$) if and only if $\opt$ satisfies Properties~\ref{prop:backward:conditioning:property}, \ref{prop:insensitivity:to:non:optimal:elements}, \ref{prop:preservation:under:addition:of:elements}, and~\ref{prop:backward:mixture:property}.
\end{theorem}

If $\backopt(\tree)=\normoper_{\opt}(\tree)$ for any consistent $\tree$, then, at decision nodes, $\backopt(\tree)$ coincides with the method of Kikuti et al.~\cite{2005:kikuti}.

\section{Discussion}\label{sec:discussion}

\subsection{Relationship Between Subtree Perfectness and Backward Induction}

Obviously, Property~\ref{prop:conditioning:property} implies Property~\ref{prop:backward:conditioning:property}, and Property~\ref{prop:mixture:property} implies Property~\ref{prop:backward:mixture:property}. Also, it is easily shown that:

\begin{lemma}\label{lem:intersection:implies:insensitivity:and:preservation}
  Property~\ref{prop:intersection:property} implies Properties~\ref{prop:insensitivity:to:non:optimal:elements} and~\ref{prop:preservation:under:addition:of:elements}.
\end{lemma}


\begin{corollary}
  If $\normoper_{\opt}$ is subtree perfect, then $\normoper_{\opt}=\backopt$.
\end{corollary}

Subtree perfectness is, however, not necessary for backward induction. For example, point-wise dominance satisfies Properties~\ref{prop:conditioning:property}, \ref{prop:insensitivity:to:non:optimal:elements}, \ref{prop:preservation:under:addition:of:elements}, and~\ref{prop:backward:mixture:property}, but as we saw in Example~\ref{ex:maximality:counterfactual}, it lacks subtree perfectness.

Backward induction does imply a weaker form of subtree perfectness. Suppose $\backopt(\tree)=\normoper_{\opt}(\tree)$ for all consistent decision trees. By definition of $\backopt$, for any node $N$ that is in at least one element of $\backopt(\tree)$ we have
\begin{equation*}
  \subtreeat{\normoper_{\opt}(\tree)}{N}=\subtreeat{\backopt(\tree)}{N}\subseteq\backopt(\subtreeat{\tree}{N})=\normoper_{\opt}(\subtreeat{\tree}{N}).
\end{equation*}

Why can this be seen as a type of subtree perfectness? A subgame perfect equilibrium point is one that induces an equilibrium point in all subgames. A subtree perfect normal form operator is one that induces its normal form solution in all subtrees. Theorem~\ref{thm:backopt:normopt:equivalence} implies \emph{every} optimal normal form \emph{decision} induces \emph{an} optimal normal form decision in any subtree. So although we do not have subtree perfectness of \emph{solutions} we do have subtree perfectness of \emph{decisions}.

\begin{definition}
  $\normoper$ is \emph{subtree perfect for normal form decisions} if for every $N$ in $\tree$ such that $N$ is in at least one element of $\normoper(\tree)$,
\begin{equation*}
  \subtreeat{\normoper(\tree)}{N}\subseteq \normoper(\subtreeat{\tree}{N}).
\end{equation*}
\end{definition}

Obviously, backward induction implies subtree perfectness for normal form decisions. But, the opposite implication does not hold: Property~\ref{prop:insensitivity:to:non:optimal:elements} is not necessary for subtree perfectness for normal form decisions.

\begin{theorem}\label{thm:weak:subtree:perfectness}
  For any choice function $\opt$, $\normoper_{\opt}$ is subtree perfect for normal form decisions if and only if $\opt$ satisfies Properties~\ref{prop:backward:conditioning:property}, \ref{prop:preservation:under:addition:of:elements}, and~\ref{prop:backward:mixture:property}.
\end{theorem}

\subsection{Normal Form and Extensive Form Equivalence}\label{sec:equivalence}

Extensive form solutions can be seen as more natural solutions to decision trees, because they represent the sequential nature of the decision making, whereas the normal form solutions remove the sequential aspects. It is common to ask whether there is an equivalence between certain normal form operators and extensive form operators. For our definitions of normal and extensive form, this is very different from the of ``normal form/extensive form coincidence'' of McClennen~\cite[p.~115]{1990:mcclennen}. In McClennen's terms, the normal form and extensive form refer to the structure of the tree, and coincidence requires the solution of the two forms to be the same. In the case of $\normoper_{\opt}$, this coincidence will always occur by definition.

The clearest equivalence arises when $\normoper(\tree)=\{\extoper(\tree)\}$, that is, when the extensive form solution is itself a normal form decision, and is also the only element of the normal form solution. Even when using choice functions corresponding to total preorders, this may not arise for every decision tree, so a more general concept of equivalence is required. We shall consider an extensive form operator and a normal form operator to be equivalent if, for every consistent decision tree $\tree$,
\begin{equation*}
  \normoper(\tree)=\nfd(\extoper(\tree)).
\end{equation*}
If such an equivalence holds, it is easy to move from the extensive form solution to the normal form solution, simply by finding all normal form decisions. It is also easy to move from the normal form to extensive form, as the following result shows.

\begin{lemma}\label{lemma:equivalent:extensive:form:construction}
  Suppose $\normoper$ and $\extoper$ are equivalent. Then a node $N$ is in $\extoper(\tree)$ if and only if $N$ is in at least one element of $\normoper(\tree)$.
\end{lemma}

If we can find equivalent normal form and extensive form operators, then either both are subtree perfect, or neither is.

\begin{lemma}\label{lemma:corresponding:operators:both:factual:or:neither}
  Suppose that $\extoper$ and $\normoper$ are equivalent. Then, $\normoper$ is subtree perfect if and only if $\extoper$ is subtree perfect.
\end{lemma}

With these results at hand, normal-extensive form equivalence is easily established, using structural induction as usual (the proof is left to the reader):

\begin{theorem}\label{thm:factual:norm:implies:corresponding:ext}
  If a normal form operator $\normoper_{\opt}$ induced by a choice function $\opt$ is subtree perfect, then there exists an equivalent subtree perfect extensive form operator $\extoper$.
\end{theorem}

Perhaps surprisingly, there are subtree perfect normal form operators (necessarily \emph{not} induced by a choice function) which have no equivalent extensive form operator. Indeed, the proof of Theorem~\ref{thm:factual:norm:implies:corresponding:ext} relies on the following consequence of Lemmas~\ref{lemma:opt:of:setsums:strong:equality} and~\ref{lemma:gamble:normopt:equivalence:chance:nodes}:
\begin{equation*}
  \textstyle\normoper_{\opt}\left(\bigchancenodemixture_{i=1}^n \event_i \tree_i \right) = \bigchancenodemixture_{i=1}^n\event_i \normoper_{\opt}(\tree_i),
\end{equation*}
General subtree perfect normal form operators need not satisfy this. For instance, the example of Section~\ref{sec:extensive:normal:form:solutions} can yield a subtree perfect normal form operator which has no equivalent extensive form representation.

\subsection{Related Work}

As mentioned earlier, our results have strong links with the work of Hammond~\cite{1988:hammond}, Machina~\cite{1989:machina}, and McClennen~\cite{1990:mcclennen}. 

Machina~\cite{1989:machina} assumes probabilities at chance nodes and choice functions that correspond to total preorders. Under these assumptions, a necessary condition for subtree perfectness of $\normoper_{\opt}$ is that $\opt$ satisfies \emph{separability over mutually exclusive events}. Consider $n$ possible events with probabilities $p_1$, \dots, $p_n$, a gamble giving reward $r_i\in\rewardset$ if event $i$ obtains, and a second gamble that gives reward $r_i$ if event $i$ occurs for $i>1$ and $r_*\in\rewardset$ for $i=1$. Separability says that the second gamble is preferred to the first if and only if $r_*$ is preferred to $r_1$.

There are various ways to adapt separability for our more general setting. For example, consider a non-trivial event $A$ and gambles $X$ and $Y$ such that $\compl{A}X=\compl{A}Y$, $AX=Ar_1$, and $AY=Ar_2$ for some rewards $r_1$ and $r_2\in\rewardset$. Separability could be: $\opt(\{X,Y\})=\{X\}$ if and only if $\opt(\{r_1, r_2\})=\{r_1\}$. Note that Property~\ref{prop:mixture:property} implies this. Indeed, Property~\ref{prop:mixture:property} can be seen as a strong form of separability, for it implies every reasonable generalization of separability.

As already noted, Hammond's~\cite{1988:hammond} results are slightly different from ours due to the definition of decision trees, and whether gambles involving probabilities are admitted. If such details are dealt with, 
then his results become very similar to ours. 
Using our terminology and notation, Hammond's first defines:

\begin{definition}
  An extensive form operator $\extoper$ is \emph{consistent} if it is subtree perfect.
\end{definition}

\begin{definition}\label{def:consequentialist}
  An extensive form operator $\extoper$ is \emph{consequentialist} if, for any decision trees $\tree_1$ and $\tree_2$ such that $\normgambles(\tree_1)=\normgambles(\tree_2)$ and $\treeevent{\tree_1}=\treeevent{\tree_2}$,
    \begin{equation*}
      \normgambles(\extoper(\tree_1))=\normgambles(\extoper(\tree_2)).
    \end{equation*}
\end{definition}

\begin{definition}
  An extensive form operator $\extoper$ is \emph{strongly consequentialist} if it is consequentialist and, for any $U\in\nfd(\tree)$ such that $\normgambles(U)\subseteq\normgambles(\extoper(\tree))$, $U\in\nfd(\extoper(\tree))$.
\end{definition}

Note that Hammond does not use the term ``strongly consequentialist''. 
He writes that consequentialism means that decisions should be valued by their gambles, which corresponds to strong consequentialism. Yet, his only mathematical definition of consequentialism seems to be identical to Definition~\ref{def:consequentialist}. The difference 
is small, but important to link our results with his.

Hammond argues that consistent and (not necessarily strongly) consequentialist extensive form operators induce a choice function on gambles as follows.

\begin{definition}
  For a consistent and consequentialist extensive form operator $\extoper$, define its corresponding choice function $\opt_{\extoper}$ by
  \begin{equation*}
    \opt_{\extoper}(\mathcal{X}|A)=\normgambles(\extoper(\tree)),
  \end{equation*}
  where $\tree$ is any consistent decision tree with $\normgambles(\tree)=\mathcal{X}$ and $\treeevent{\tree}=A$. Because $\extoper$ is consistent and consequentialist, this choice function exists and does not depend on the choice of $\tree$.
\end{definition}

With these definitions, we can prove a slightly stronger version of Hammond's results~\cite[Theorem~5.4, Theorem~6, Theorem~7, and Theorem~8]{1988:hammond}.

\begin{theorem}\label{theorem:strong:hammond:theorem}
  A choice function $\opt$ satisfies Properties~\ref{prop:conditioning:property}, \ref{prop:intersection:property}, and~\ref{prop:mixture:property} if and only if there is a consistent and strongly consequentialist extensive form operator $\extoper$ such that $\opt_{\extoper}=\opt$.
\end{theorem}

\begin{proof}
  ``if''. Follow the approach of Hammond~\cite[Theorem~5.4 and Theorem~7]{1988:hammond}. Note that these proofs require only consequentialism.
  
  ``only if''. Suppose $\opt$ satisfies Properties~\ref{prop:conditioning:property}, \ref{prop:intersection:property}, and~\ref{prop:mixture:property}. By Theorems~\ref{thm:conditions:to:be:factual} and~\ref{thm:factual:norm:implies:corresponding:ext}, $\normoper_{\opt}$ is subtree perfect and has an equivalent subtree perfect extensive form operator $\extoper$. 
By definition of $\normoper_{\opt}$, for any strategically equivalent trees $\tree_1$ and $\tree_2$, $\normgambles(\normoper_{\opt}(\tree_1))=\normgambles(\normoper_{\opt}(\tree_2))$, and so the same holds for $\extoper$. Hence, $\extoper$ is consistent and consequentialist. By construction, $\extoper$ is also strongly consequentialist, and obviously also $\opt_{\extoper}=\opt$.
\end{proof}

It is easily seen that, for a particular choice function $\opt$, there is exactly one consistent and strongly consequentialist extensive form operator that induces $\opt$. Therefore, there is an equivalence between the consistent and \emph{strongly} consequentialist extensive form operator inducing $\opt$ and the subtree perfect normal form operator induced by $\opt$. This equivalence is \emph{not} present in Hammond's account, since 
multiple consistent and (\emph{not strongly}) consequentialist extensive form operators can induce the same choice function.

For example, if $\tree$ has two normal form decisions inducing the same gamble $X$, $\nfd(\extoper_1(\tree))$ could only include one, while $\nfd(\extoper_2(\tree))$ includes both, without violating subtree perfectness. Moreover, $\normgambles(\tree)$ contains $X$ so strategic equivalence is preserved as required. Hence, both operators can be consistent and consequentialist, 
however
only at most one of these can be equivalent to $\normoper_{\opt}$. Indeed, 
a consistent and consequentialist extensive form operator can \emph{only} be equivalent to its corresponding $\normoper_{\opt}$ if it is \emph{strongly} consequentialist.

As with Hammond, McClennen's \cite{1990:mcclennen} decision trees differ in that some chance nodes can have probabilities for events. 
Also, there seems to be no concept of conditioning in McClennen's account. As noted in Section~\ref{sec:forms}, McClennen's dynamic normal form solutions are more general types of normal form solutions. Some of his results \cite[Theorems~8.1 and~8.2]{1990:mcclennen} are similar to ours, and are based on three restrictions placed on his solutions. 

\begin{definition}[McClennen, \protect{\cite[p.~120]{1990:mcclennen}}]
  A dynamic normal form solution satisfies \emph{dynamic consistency} if, for every node $N$ in the tree, the restriction of the optimal set at the root node to $N$ is exactly the optimal set at $N$.
\end{definition}

Obviously, there is a one-to-one correspondence between dynamic normal form solutions satisfying dynamic consistency, and our normal form solutions. 

\begin{definition}[McClennen,\protect{\cite[p.~114]{1990:mcclennen}}]
  A dynamic normal form solution satisfies \emph{plan reduction} if, for every normal form decision in $\tree$ that induces the same gamble, either all or none of them are optimal.
\end{definition}

By definition, $\normoper_{\opt}$ satisfies plan reduction.

\begin{definition}[McClennen,\protect{\cite[p.~122]{1990:mcclennen}}]
  A dynamic normal form solution satisfies \emph{separability} if, for any tree $\tree$ and any node $N$ in $\tree$, the set of optimal plans at $N$ is the same as the set of optimal plans of the separate tree $\subtreeat{\tree}{N}$.
\end{definition}

On its own, separability is not exactly subtree perfectness, but if dynamic consistency holds then the two properties are equivalent. McClennen's two theorems can then be adapted into our setting as:

\begin{theorem}
  If a dynamic normal form solution satisfies plan reduction, dynamic consistency, and separability, then then it coincides with $\normoper_{\opt}$ for a choice function $\opt$ satisfying Properties~\ref{prop:conditioning:property}, \ref{prop:intersection:property}, and~\ref{prop:mixture:property}.
\end{theorem}

\begin{proof}
  The proof is essentially identical to that of Lemma~\ref{lemma:conditioning:intersection:mixture:are:necessary:for:factuality}.
\end{proof}

\section{Conclusion}

We extended Selten's idea of subtree perfectness to decision trees, for normal and extensive for solutions. Subtree perfectness for extensive form solutions is Hammond's consistency condition. Subtree perfectness for normal form solutions is, under the assumption of dynamic consistency, McClennen's separability. We found necessary and sufficient conditions for a choice function to induce a subtree perfect normal form operator. These turned out to be similar to, but stronger than, those for backward induction to work. So, even if a normal form operator lacks subtree perfectness, it may still be possible to find the normal form solution by backward induction.

While many choice functions satisfy Property~\ref{prop:conditioning:property}, Properties~\ref{prop:intersection:property} and~\ref{prop:mixture:property} are perhaps more restrictive than one would like. Is violating subtree perfectness acceptable? We believe that subtree perfectness is a desirable property and one must think carefully before abandoning it. On the other hand, if one is attracted to the three properties for other reasons, then subtree perfectness gives them a strong justification, particularly since (at least for Properties~\ref{prop:intersection:property} and~\ref{prop:mixture:property}) they are much more difficult to justify in a static setting. Attempts to justify violation of subtree perfectness, without violating dynamic consistency, have however been made, for example by Machina~\cite{1989:machina} and McClennen~\cite[9.6]{1990:mcclennen}.

We recovered the well-known fact that subtree perfectness requires total preordering. Many choice functions suggested in the literature, such as maximality \cite[Sec.~3.9]{1991:walley} and E-admissibility \cite{1980:levi}, violate total preordering, and hence fail subtree perfectness. Interestingly, we can easily establish that some of these (particularly, maximality and E-admissibility) still admit backward induction.

If one is committed to the idea of subtree perfectness but also wishes to use a choice function that fails Property~\ref{prop:intersection:property}, then the best solution may be to use an extensive form operator. One can easily define subtree perfect normal form solutions based on choice functions (but not directly induced as in Definition~\ref{def:normoper}). These, however, can have unpleasant behaviour such as admitting pointwise dominated options. It is easier to avoid such behaviour with subtree perfect extensive form solutions, as proposed for instance by Seidenfeld \cite{1988:seidenfeld}, although Seidenfeld's idea will only work if the choice function can model complete ignorance. Roughly, this solution is found by backward induction, and when a decision node admits multiple optimal options, then it is treated as a chance node with complete ignorance about which of the optimal decisions is chosen.

We have seen that, when 
using choice functions on gambles, subtree perfectness is closely related to equivalence between extensive form and normal form operators, both using our approach of defining an operator based on a choice function, and Hammond's approach of defining a choice function based on an operator. Interestingly, normal-extensive form equivalence need \emph{not} hold for subtree perfect operators that are not induced by choice functions.

Although we have primarily investigated $\normoper_{\opt}$ in this paper, we do not argue that a normal form operator, and $\normoper_{\opt}$ in particular, gives the best solution to a decision tree. A normal form solution requires a policy for all eventualities to be specified and adhered to. The subject adheres to this policy only by his own resolution: he may of course have the ability to change his policy upon reaching a decision node \cite{1988:seidenfeld}. One could therefore argue that a normal form solution is only acceptable for sequential problems when the subject does not get the chance to change his mind (for example, if he instructs, in advance, others to carry out the actions).

In practice, applying $\opt$ to the set of all normal form gambles may be difficult. Therefore, we defined a normal form operator which yields a normal form solution by means of backward induction, and which will, in many cases, be easier to apply than $\normoper_{\opt}$. We found necessary and sufficient conditions on $\opt$ for our backward induction algorithm to yield exactly $\normoper_{\opt}$. As mentioned, we found that total preordering is not necessary for backward induction to work: the set of choice functions that satisfy the backward induction properties are a subset of those that satisfy subtree perfectness for normal form solutions and a superset of those that satisfy subtree perfectness. Hence, it remains unclear whether backward induction has any justification beyond practicality.

\section*{Acknowledgements}

The authors are indebted to Teddy Seidenfeld for suggesting the term `subtree perfectness', and to Wlodek Rabinowicz for suggesting the link between separability and subtree perfectness. EPSRC supports the first author.

\bibliography{imprecisetrees} 
\bibliographystyle{plain}

\newpage

\appendix

\section{Proofs}

\subsection{Proof of Eq.~\eqref{eq:normgambles:tree:is:normgambles:nfd:tree}}

\begin{lemma}\label{lemma:gambisgambnfd}
  For any decision tree $\tree$, $\normgambles(\tree)=\normgambles(\nfd(\tree))$.
\end{lemma}
\begin{proof}
  We prove this by structural induction. In the base step, we prove the equality for trees comprising only one node. In the induction step, we prove that if the equality holds for the subtrees at every child of the root node of $\tree$, then the equality also holds for $\tree$.
  
  If $\tree$ consists of only a single node, namely a reward node, then $\nfd(\tree)=\{\tree\}$ and the result holds trivially. Thus the base step is confirmed.
  
  Suppose $\tree$ has a chance node at the root, that is, $\tree=\bigchancenodemixture_{i=1}^n \event_i \tree_i$. Each element of $\nfd(\tree)$ is of the form $\bigchancenodemixture_{i=1}^n \atree_i$, where $\atree_i \in \nfd(\tree_i)$. In other words, $\nfd(\tree)$ is the set of all possible mixtures of the elements of $\nfd(\tree_i)$, that is,
  \begin{equation*}
    \nfd(\tree)=\bigchancenodemixture_{i=1}^n \event_i \nfd(\tree_i).
  \end{equation*}
  The induction hypothesis is $\normgambles(\tree_i)=\normgambles(\nfd(\tree_i))$ for each $i$. We have
  \begin{align*}
    \normgambles(\nfd(\tree)) &= \normgambles\Bigg(\bigchancenodemixture_{i=1}^n \event_i \nfd(\tree_i)\Bigg) \\
    &= \biggambplus_{i=1}^n \event_i \normgambles(\nfd(\tree_i)) \\
    &= \biggambplus_{i=1}^n \event_i \normgambles(\tree_i) \\
    &= \normgambles(\tree).
  \end{align*}
  
  On the other hand, if $\tree$ has a decision node as a root, that is, $\tree=\bigdecnodeunion_{i=1}^n \tree_i$, then
  \begin{equation*}
    \nfd(\tree)= \Bigg\{\decnodeunion U\colon U \in \bigcup_{i=1}^n \nfd(\tree_i)\Bigg\},
  \end{equation*}
  and, since $\normgambles(\decnodeunion U) = \normgambles(U)$ for any $U$, 
  \begin{align*}
    \normgambles(\nfd(\tree))  &= \Bigg\{\normgambles(\decnodeunion U)\colon U\in\bigcup_{i=1}^n \nfd(\tree_i)\Bigg\}\\
    &= \Bigg\{\normgambles(U)\colon U\in\bigcup_{i=1}^n \nfd(\tree_i)\Bigg\}\\
    &= \bigcup_{i=1}^n \normgambles(\nfd(\tree_i)) \\
    \intertext{and again, the induction hypothesis says that $\normgambles(\tree_i)=\normgambles(\nfd(\tree_i))$ for each $i$, so}
    &= \bigcup_{i=1}^n \normgambles(\tree_i) \\
    &= \normgambles(\tree).
  \end{align*}
  This completes the induction step.
\end{proof}

\subsection{Equivalence of Definition~\ref{def:gambles:consistent}\ref{def:gambles:consistent:via:trees} and~\ref{def:gambles:consistent:via:inverse:map}}

\begin{proof}[Proof of equivalence]
  \ref{def:gambles:consistent:via:trees}$\implies$\ref{def:gambles:consistent:via:inverse:map}. We prove the implication by structural induction on the tree $\tree$. In the base step, we prove that, for every $\mathcal{X}$ and non-empty event $A$, the implication holds for consistent decision trees which consist of only a single node. In the induction step, we prove that if, for all $\mathcal{X}$ and non-empty $A$, the implication holds for the subtrees at every child of the root node, then, for all $\mathcal{X}$ and non-empty $A$, the implication also holds for the whole tree.

  First, if $\tree$ consists of only a single node, namely a reward node, then $\normgambles(\tree)=\{s\}$ for some $s\in\rewardset$, so by assumption, $\mathcal{X}=\normgambles(\tree)=\{X\}$ where $X$ is the gamble yielding a constant value $s$. Clearly, $X^{-1}(r)=\emptyset$ for all $r\neq s$ and $X^{-1}(s)=\pspace$, hence indeed $X^{-1}(r)\cap A=A\neq\emptyset$ whenever $X^{-1}(r)\neq\emptyset$, and this for every non-empty event $A$.
  
  Next, suppose $\tree$ has a chance node as its root, that is, $\tree=\bigchancenodemixture_{i=1}^n E_i \tree_i$. By assumption, and by the definition of $\normgambles$ (see Definition~\ref{def:normgambles}),
  \begin{equation*}
    \mathcal{X}=\normgambles(\tree)=\biggambplus_{i=1}^n \event_i\normgambles(\tree_i)=\biggambplus_{i=1}^n \event_i\mathcal{X}_i
  \end{equation*}
  where $\mathcal{X}_i$ is a shorthand notation for $\normgambles(\tree_i)$. Also, by assumption, $\treeevent{\tree}=A$. By the induction hypothesis, we know that for every $r\in\rewardset$ and every $X_i\in\mathcal{X}_i$ such that $X_i^{-1}(r)\neq\emptyset$, it holds that $X_i^{-1}(r)\cap A\cap E_i\neq\emptyset$ (indeed, $\treeevent{\tree_i}=\treeevent{\tree}\cap E_i=A\cap E_i$). Now we have all ingredients to prove the desired implication. Indeed, for any $X\in\mathcal{X}$, or equivalently, for any $X=\biggambplus_{i=1}^n\event_i X_i$ with $X_i\in\mathcal{X}_i$,
  \begin{equation}
    \label{eq:def:gambles:consistent:helper:1}
    X^{-1}(r)
    =\left(\biggambplus_{i=1}^n\event_i X_i\right)^{-1}(r)
    =\bigcup_{i=1}^n X_i^{-1}(r)\cap\event_i
  \end{equation}
  Hence Eq.~\eqref{eq:def:gambles:consistent:helper:1} implies that $X^{-1}(r)\neq\emptyset$ whenever $X_i^{-1}(r)\cap\event_i\neq\emptyset$ for at least one $i\in\{1,\dots,n\}$, and in that case, it obviously follows that also $X_i^{-1}(r)\neq\emptyset$, which implies, as just shown, that $X_i^{-1}(r)\cap A\cap E_i\neq\emptyset$. But, then, again by Eq.~\eqref{eq:def:gambles:consistent:helper:1}, it must also hold that $X^{-1}(r)\cap A\neq\emptyset$.

  Finally, suppose that the root of $\tree$ is a decision node, that is, $\tree=\bigdecnodeunion_{i=1}^n \tree_i$. By assumption, and by the definition of $\normgambles$ (see Definition~\ref{def:normgambles}),
  \begin{equation*}
    \mathcal{X}=\normgambles(\tree)=\bigcup_{i=1}^n \normgambles(\tree_i)=\bigcup_{i=1}^n\mathcal{X}_i
  \end{equation*}
  where $\mathcal{X}_i$ is a shorthand notation for $\normgambles(\tree_i)$. Also, by assumption, $\treeevent{\tree}=A$. By the induction hypothesis, we know that for every $r\in\rewardset$ and every $X_i\in\mathcal{X}_i$ such that $X_i^{-1}(r)\neq\emptyset$, it holds that $X_i^{-1}(r)\cap A\neq\emptyset$ (indeed, $\treeevent{\tree_i}=\treeevent{\tree}=A$). But, for any $X\in\mathcal{X}$, it follows that $X=X_i$ for some $X_i\in\mathcal{X}_i$, so the desired implication follows immediately from the induction hypothesis.

  \ref{def:gambles:consistent:via:inverse:map}$\implies$\ref{def:gambles:consistent:via:trees}. Suppose that for every $r\in\rewardset$ and every $X\in\mathcal{X}$ such that $X^{-1}(r)\neq\emptyset$, it holds that $X^{-1}(r)\cap A\neq\emptyset$. Consider the decision tree
  \begin{equation*}
    \tree=\bigdecnodeunion_{X\in\mathcal{X}}\bigchancenodemixture_{\substack{r\in\rewardset \\ X^{-1}(r)\neq\emptyset}}X^{-1}(r)r
  \end{equation*}
  with $\treeevent{\tree}=A$. Let $N(X)$ denote the chance node of $\tree$ associated with $X$, and let $N(X,r)$ denote the reward node of $\tree$ associated with $X$ and $r$ (of course $N(X,r)$ only exists for $X^{-1}(r)\neq\emptyset$, by definition of $\tree$).

  Clearly, $\tree$ is consistent, because $\treeevent{\subtreeat{\tree}{N(X)}}=\treeevent{\tree}=A\neq\emptyset$ and $\treeevent{\subtreeat{\tree}{N(X,r)}}=\treeevent{\subtreeat{\tree}{N(X)}}\cap X^{-1}(r)=A\cap X^{-1}(r)\neq\emptyset$ by assumption, and
  \begin{align*}
    \normgambles(\tree)
    &=\bigcup_{X\in\mathcal{X}}\normgambles\left(\bigchancenodemixture_{\substack{r\in\rewardset \\ X^{-1}(r)\neq\emptyset}}X^{-1}(r) r\right)
    \\
    &
    =\bigcup_{X\in\mathcal{X}}\left\{\biggambplus_{\substack{r\in\rewardset \\ X^{-1}(r)\neq\emptyset}}X^{-1}(r) r\right\}
    =\bigcup_{X\in\mathcal{X}}\{X\}=\mathcal{X}
  \end{align*}
  which establishes \ref{def:gambles:consistent:via:trees}.
\end{proof}

\subsection{Proof of Lemma~\ref{lemma:strong:path:independence:equivalence}}

\begin{proof}
  In this proof, $A$ is a non-empty event and all gambles are $A$-consistent.
  
  Property~\ref{prop:intersection:property}$\implies$Property~\ref{prop:very:strong:path:independence}.
  Let $\mathcal{X}_1$, \dots ,$\mathcal{X}_n$ be non-empty finite sets of gambles, and let $\mathcal{X}=\bigcup_{i=1}^n \mathcal{X}_i$. If $\opt(\mathcal{X}|A)\cap\mathcal{X}_k\neq\emptyset$, then $\opt(\mathcal{X}_k|A)=\opt(\mathcal{X}|A)\cap\mathcal{X}_k$. Hence,
  \begin{equation*}
    \opt(\mathcal{X}|A)
    =\bigcup_{k=1}^n\opt(\mathcal{X}|A)\cap\mathcal{X}_k
    =\bigcup_{\substack{k=1\\\mathcal{X}_k\cap\opt(\mathcal{X}|A)\neq\emptyset}}^n\opt(\mathcal{X}_k|A)
  \end{equation*}

  Property~\ref{prop:very:strong:path:independence}$\implies$Property~\ref{prop:strong:path:independence}. Immediate.
  
  Property~\ref{prop:strong:path:independence}$\implies$Property~\ref{prop:total:order:property}. Define $X\succeq_A Y$ if $X\in\opt(\{X,Y\}|A)$. First, we prove that $\succeq_A$ is a total preorder (i.e. total, reflexive, and transitive). Clearly, $\succeq_A$ is total since $X\in\opt(\{X,Y\}|A)$ or $Y\in\opt(\{X,Y\}|A)$, hence $X\succeq_A Y$ or  $Y\succeq_A X$, for all normal form decisions $X$ and $Y$. Obviously, $\succeq_A$ is reflexive. Is $\succeq_A$ transitive? Suppose $X\succeq_A Y$ and $Y\succeq_A Z$.

  By Property~\ref{prop:strong:path:independence},
  \begin{equation*}
    \opt(\{X,Y,Z\}|A)
    =
    \begin{cases}
      \opt(\{Y,Z\}|A), & \text{or} \\
      \{X\}, & \text{or} \\
      \opt(\{Y,Z\}|A)\cup\{X\}.
    \end{cases}
  \end{equation*}
  Since, $Y\succeq_A Z$, it follows that $\{X,Y\}\cap\opt(\{X,Y,Z\}|A)\neq\emptyset$.

  Again, by Property~\ref{prop:strong:path:independence},
  \begin{equation*}
    \opt(\{X,Y,Z\}|A)
    =
    \begin{cases}
      \opt(\{X,Y\}|A), & \text{or} \\
      \{Z\}, & \text{or} \\
      \opt(\{X,Y\}|A)\cup\{Z\}.
    \end{cases}
  \end{equation*}
  The case $\opt(\{X,Y,Z\}|A)=\{Z\}$ cannot occur however, because we just showed that $\{X,Y\}\cap\opt(\{X,Y,Z\}|A)\neq\emptyset$. Hence, because $X\in\opt(\{X,Y\}|A)$, it follows that $X\in\opt(\{X,Y,Z\}|A)$.

  Once more by Property~\ref{prop:strong:path:independence},
  \begin{equation*}
    \opt(\{X,Y,Z\}|A)
    =
    \begin{cases}
      \opt(\{X,Z\}|A), & \text{or} \\
      \{Y\}, & \text{or} \\
      \opt(\{X,Z\}|A)\cup\{Y\}.
    \end{cases}
  \end{equation*}
  We just showed that $X\in\opt(\{X,Y,Z\}|A)$, hence the second case cannot occur, and it can only be that also $X\in\opt(\{X,Z\}|A)$, establishing $X\succeq_A Z$.

  Finally, we prove that
  \begin{align*}
    \opt(\mathcal{X}|A)
    =
    \{X\in\mathcal{X}\colon(\forall Y\in\mathcal{X})(X\succeq_A Y)\},
  \end{align*}
  or equivalently, we prove for any $X\in\mathcal{X}$ that $X\in\opt(\mathcal{X}|A)$ if and only if $X\in\opt(\{X,Y\}|A)$ for all $Y\in\mathcal{X}$.

  Indeed, by Property~\ref{prop:strong:path:independence}, for any $X$ and $Y$ in $\mathcal{X}$, it holds that
  \begin{equation*}
    \opt(\mathcal{X}|A)=
    \begin{cases}
      \opt(\{X,Y\}|A), & \text{or} \\
      \opt(\mathcal{X}\setminus\{X,Y\}|A), & \text{or} \\
      \opt(\{X,Y\}|A)\cup\opt(\mathcal{X}\setminus\{X,Y\}|A).
    \end{cases}
  \end{equation*}
  and hence, if $X\in\opt(\mathcal{X}|A)$ then the second option is impossible and therefore $X\in\opt(\{X,Y\}|A)$ for all $Y\in\mathcal{Y}$.

  Conversely, again by Property~\ref{prop:strong:path:independence}
  \begin{equation*}
    \opt(\mathcal{X}|A)
    =\opt\Bigg(\bigcup_{Y\in\mathcal{X}}\{X,Y\}\Bigg|A\Bigg)
    =\bigcup_{Y\in\mathcal{Y}}\opt(\{X,Y\}|A)
  \end{equation*}
  for some subset $\mathcal{Y}$ of $\mathcal{X}$, and hence,
  if $X\in\opt(\{X,Y\}|A)$ for all $Y\in\mathcal{X}$, then $X\in\opt(\mathcal{X}|A)$.

  Property~\ref{prop:total:order:property}$\implies$Property~\ref{prop:intersection:property}. Assume that $\opt(\mathcal{X}|A)\cap\mathcal{Y}\neq\emptyset$. This means that there must be an $Y^*\in\mathcal{Y}$ such that $Y^*\succeq_A X$ for all $X\in\mathcal{X}$. Clearly, $Y^*\in\opt(\mathcal{Y}|A)$. But, for all $Y\in\opt(\mathcal{Y}|A)$ it must also hold that $Y\succeq_A Y^*$, and hence $Y\succeq_A X$ for all $X\in\mathcal{X}$ as $\succeq_A$ is transitive. We conclude:
  \begin{equation*}
    \opt(\mathcal{Y}|A)
    =
    \{Y\in\mathcal{Y}\colon(\forall X\in\mathcal{X})(Y\succeq_A X)\}
    =
    \opt(\mathcal{X}|A)\cap\mathcal{Y}
  \end{equation*}
\end{proof}

\subsection{Proof of Lemma~\ref{lemma:factuality:at:children}}

\begin{proof}
  If $N$ is the root of $\tree$, then the statement is trivial. If $N\in\children(\tree)$, then the statement follows from \ref{lemma:factuality:at:children:childcondition}. Otherwise, $N$ must belong to $\subtreeat{\tree}{K}$ for some $K\in\children(\tree)$.

  First, note that $K$ is a node of at least one element of $\normoper(\tree)$. Indeed, it is given that $N$ is a node of at least one element, say $\atree$, of $\normoper(\tree)$. Then, obviously, $K$ must also be a node of $\atree$, simply because any node on the unique path within $\tree$ between the root of $\tree$ and $N$ must be a node of $\atree$, and one of those nodes is $K$. So, $K$ is a node of an element of $\normoper(\tree)$ (namely, $U$).

  Secondly, note that $N$ is also a node of at least one element of $\normoper(\subtreeat{\tree}{K})$. Indeed, $N$ is a node of an element of $\normoper(\tree)$, and hence, in particular also of $\subtreeat{\normoper(\tree)}{K}$. But, by \ref{lemma:factuality:at:children:childcondition}, and the fact that $K$ is a node of at least one element of $\normoper(\tree)$ (as just proven), it follows that $\subtreeat{\normoper(\tree)}{K}=\normoper(\subtreeat{\tree}{K})$. Hence, $N$ is also a node of at least one element of $\normoper(\subtreeat{\tree}{K})$.

  Combining everything, it follows that
  \begin{align*}
    \normoper(\subtreeat{\tree}{N})
    &=\normoper(\subtreeat{\subtreeat{\tree}{K}}{N}) \\
    \intertext{so, by \ref{lemma:factuality:at:children:grandchildcondition}, and because $N$ is in at least one element of $\normoper(\subtreeat{\tree}{K})$,}
    &=\subtreeat{\normoper(\subtreeat{\tree}{K})}{N} \\
    \intertext{hence, by \ref{lemma:factuality:at:children:childcondition}, and since $K$ is in at least one element of $\normoper(\tree)$,}
    &=\subtreeat{\subtreeat{\normoper(\tree)}{K}}{N}
    =\subtreeat{\normoper(\tree)}{N}.
  \end{align*}
\end{proof}

\subsection{Proof of Lemma~\ref{lemma:opt:of:setsums:strong:equality}}

\begin{proof}
  The statement is trivial if $n=1$ (because, in that case, $A_1=\pspace$). Let us prove the statement also in case $n\ge 2$.
  
  Let $\mathcal{X}=\biggambplus_{i=1}^nA_i\mathcal{X}_i$. Consider any $k\in\{1,\dots,n\}$ and let $$\mathcal{Z}_k=\biggambplus_{j\neq k} A'_j \mathcal{X}_j$$
where $(A'_j)_{j\neq k}$ forms an arbitrary partition of $\pspace$ such that $\compl{A}_k\cap A'_j=A_j$ for all $j\neq k$. Clearly, $\mathcal{Z}_k$ is $\compl{A}_k\cap B$-consistent because we can trivially find a consistent decision tree $\tree$ with $\treeevent{\tree}=\compl{A}_k\cap B$ and $\normgambles(\tree)=\mathcal{Z}_k$, using the $A_j \cap B$-consistency (and hence, $\compl{A}_k\cap A'_j \cap B$-consistency) of each $\mathcal{X}_j$ for $j\neq k$.

Now, observe that by construction of $\mathcal{Z}_k$,
\begin{equation*}
\mathcal{X}=A_k \mathcal{X}_k \gambplus \compl{A}_k \mathcal{Z}_k
 =\bigcup_{Z_k\in\mathcal{Z}_k} (A_k \mathcal{X}_k \gambplus \compl{A}_k Z_k).
\end{equation*}
Note that $\mathcal{X}$ is $B$-consistent (indeed, because each $\mathcal{X}_i$ is $A_i \cap B$-consistent, we can trivially find a consistent decision tree $\tree$ with $\treeevent{\tree}=B$ and $\normgambles(\tree)=\mathcal{X}$).

  Since Property~\ref{prop:intersection:property} holds, Property~\ref{prop:strong:path:independence} holds as well by Lemma~\ref{lemma:strong:path:independence:equivalence}. So, if we apply $\opt(\cdot|B)$ on both sides of the above equality, then it follows from Property~\ref{prop:strong:path:independence} that
  \begin{align}
    \nonumber
    \opt(\mathcal{X}|B)
    &=
    \nonumber
    \bigcup_{Z_k\in\mathcal{Z}_k^*}
    \opt(A_k\mathcal{X}_k\gambplus\compl{A}_k Z_k|B),\\
    \intertext{
      for some $\mathcal{Z}_k^*\subseteq\mathcal{Z}_k$.
      By Property~\ref{prop:mixture:property},}
    &=
    \nonumber
    \bigcup_{Z_k\in\mathcal{Z}_k^*}
    (A_k\opt(\mathcal{X}_k|A_k\cap B) \gambplus \compl{A}_k Z)\\
    &=
    \label{eq:opt:of:setsums:strong:equality:help:1}
    A_k\opt(\mathcal{X}_k|A_k\cap B) \gambplus \compl{A}_k\mathcal{Z}_k^*
  \end{align}
  Since this holds for each $k\in\{1,\dots,n\}$, we arrive at Eq.~\eqref{eq:opt:of:setsums:strong:equality}, by Lemma~\ref{lemma:setsums:equality}.
\end{proof}

We used the following lemma.

\begin{lemma}\label{lemma:setsums:equality}
  Let $A_1$, \dots, $A_n$ be a finite partition of $\pspace$, $n\ge 2$. Let $\mathcal{X}$, and $\mathcal{X}_1$, $\mathcal{Z}_1$, \dots, $\mathcal{X}_n$, $\mathcal{Z}_n$ be non-empty finite sets of gambles. If
  \begin{equation*}
    \mathcal{X}=A_k\mathcal{X}_k\gambplus\compl{A}_k \mathcal{Z}_k
  \end{equation*}
  for all $k\in\{1,\dots,n\}$, then
  \begin{equation*}
    \mathcal{X}=\biggambplus_{i=1}^nA_i\mathcal{X}_i.
  \end{equation*}
\end{lemma}
\begin{proof}
  Let us prove the implication by induction.

  The implication holds for $n=2$. Indeed, suppose that
  \begin{equation*}
    \mathcal{X}=A_1\mathcal{X}_1\gambplus\compl{A}_1\mathcal{Z}_1=A_2\mathcal{X}_2\gambplus\compl{A}_2\mathcal{Z}_2.
  \end{equation*}
  By multiplying all elements of these sets with $A_2$, and noting that $\compl{A}_1=A_2$ and $\compl{A}_2=A_1$, we find that
  \begin{equation*}
    \compl{A}_1\mathcal{Z}_1=A_2\mathcal{X}_2
  \end{equation*}
  which establishes the base step.

  For the induction step, assume that the implication holds for $n=m$. We prove that the implication also holds for $n=m+1$. Suppose that
  \begin{align*}
     \mathcal{X}&=A_1\mathcal{X}_1\gambplus\compl{A}_1\mathcal{Z}_1 \\
     \mathcal{X}&=A_2\mathcal{X}_2\gambplus\compl{A}_2\mathcal{Z}_2 \\
     &\quad\vdots \\
     \mathcal{X}&=A_{m+1}\mathcal{X}_{m+1}\gambplus\compl{A}_{m+1}\mathcal{Z}_{m+1}
  \end{align*}
  First, note that the equality $A_1\mathcal{X}_1\gambplus\compl{A}_1\mathcal{Z}_1=A_2\mathcal{X}_2\gambplus\compl{A}_2\mathcal{Z}_2$
  implies in particular that (by multiplying all elements of these sets with $\compl{A}_1$)
  \begin{equation*}
    \compl{A}_1\mathcal{Z}_1=A_2\mathcal{X}_2\gambplus(\compl{A}_1\cap\compl{A}_2)\mathcal{Z}_2
  \end{equation*}
  Hence,
  \begin{align*}
    \mathcal{X}
    &=A_1\mathcal{X}_1\gambplus\compl{A}_1\mathcal{Z}_1
    =A_1\mathcal{X}_1\gambplus\compl{A}_1(A_2\mathcal{X}_2\gambplus(\compl{A}_1\cap\compl{A}_2)\mathcal{Z}_2) \\
    &=A_1\mathcal{X}_1\gambplus A_2\mathcal{X}_2\gambplus(\compl{A}_1\cap\compl{A}_2)\mathcal{Z}_2 \\
    &=(A_1\cup A_2)(A_1\mathcal{X}_1\gambplus A_2\mathcal{X}_2)\gambplus(\compl{A}_1\cap\compl{A}_2)\mathcal{Z}_2
  \end{align*}
  With $A^*_1=A_1\cup A_2$, $\mathcal{X}^*_1=A_1\mathcal{X}_1\gambplus A_2\mathcal{X}_2$, and $\mathcal{Z}^*_1=\mathcal{Z}_2$, we now have that
  \begin{align*}
     \mathcal{X}&=A^*_1\mathcal{X}^*_1\gambplus\compl{A}^*_1\mathcal{Z}^*_1 \\
     \mathcal{X}&=A_3\mathcal{X}_3\gambplus\compl{A}_3\mathcal{Z}_3 \\
     &\quad\vdots \\
     \mathcal{X}&=A_{m+1}\mathcal{X}_{m+1}\gambplus\compl{A}_{m+1}\mathcal{Z}_{m+1}
  \end{align*}
  By assumption, the implication to be proven already holds for $n=m$, so, from the above equalities, it follows that
  \begin{align*}
    \mathcal{X}&=A^*_1\mathcal{X}^*_1\gambplus A_3\mathcal{X}_3\gambplus\dots\gambplus A_{m+1}\mathcal{X}_{m+1}
    \\
    \intertext{and, by construction of $A^*_1$ and $\mathcal{X}^*_1$,}
    &=\biggambplus_{i=1}^{m+1}A_i\mathcal{X}_i
  \end{align*}
  which proves the induction step.
\end{proof}

\subsection{Proof of Lemma~\ref{lemma:I:Istar:equality}}

\begin{proof}
  By Eq.~\eqref{eq:normgambles:tree:is:normgambles:nfd:tree}, Eq.~\eqref{eq:lemma:I:Istar:equality} holds if and only if
  \begin{equation*}
    \normgambles(\nfd(\tree_i)) \cap \opt(\normgambles(\tree)|\treeevent{T})\neq\emptyset,
  \end{equation*}
or equivalently, if and only if there is a normal form decision $U\in\nfd(\tree_i)$ such that $\normgambles(U)\subseteq\opt(\normgambles(\tree)|\treeevent{T})$ (remember that $\normgambles(U)$ is a singleton).

  But, by definition of the $\normgambles$ operator, it also holds that $\normgambles(U)=\normgambles(\decnodeunion U)$. Hence, Eq.~\eqref{eq:lemma:I:Istar:equality} holds if and only if there is a normal form decision $U\in\nfd(\tree_i)$ such that $\normgambles(\decnodeunion U)\subseteq\opt(\normgambles(\tree)|\treeevent{T})$, or equivalently, if and only if there is a normal form decision $V\in\nfd(\tree)$ which contains the node $N_i$, such that $\normgambles(V)\subseteq\opt(\normgambles(\tree)|\treeevent{T})$.

By definition of $\normoper_{\opt}$, this is equivalent to stating that Eq.~\eqref{eq:lemma:I:Istar:equality} holds if and only if $N_i$ is in at least one element of $\normoper_{\opt}(\tree)$.
\end{proof}

\subsection{Proof of Lemma~\ref{lemma:gamble:normopt:equivalence:chance:nodes}}

\begin{proof}
  Assume that Eq.~\eqref{lemma:gamble:normopt:equivalence:chance:nodes:helper} holds.

  First, consider a normal form decision $U\in\bigchancenodemixture_{i=1}^nE_i\normoper_{\opt}(\tree_i)$. Obviously,
  \begin{align*}
    \normgambles(U) &\subseteq \normgambles\Bigg(\bigchancenodemixture_{i=1}^n E_i \normoper_{\opt}(\tree_i)\Bigg)  \\
    \intertext{so, by the definition of $\normgambles$, Eq.~\eqref{eq:normgambles:chancenodes:for:sets:of:trees} in particular,}
    &= \biggambplus_{i=1}^n E_i \normgambles(\normoper_{\opt}(\tree_i)) \\
    \intertext{
      and hence, by Eq.~\eqref{lemma:gamble:normopt:equivalence:chance:nodes:helper},}
    &= \normgambles(\normoper_{\opt}(\tree)).
  \end{align*}
  So, there exists a normal form decision $V\in\normoper_{\opt}(\tree)$ such that $\normgambles(V)=\normgambles(U)$. Since $U\in\nfd(\tree)$, by definition of $\normoper_{\opt}$ we have $U\in\normoper_{\opt}(\tree)$. So we have shown that
  \begin{equation*}
    \normoper_{\opt}(\tree)\supseteq\bigchancenodemixture_{i=1}^nE_i\normoper_{\opt}(\tree_i).
  \end{equation*}

  Next, consider a normal form decision $U\in\normoper_{\opt}(\tree)$. We know by Eq.~\eqref{lemma:gamble:normopt:equivalence:chance:nodes:helper} that
  \begin{equation*}
    \normgambles(U) \subseteq \biggambplus_{i=1}^n E_i \normgambles(\normoper_{\opt}(\tree_i)).
  \end{equation*}
  We can write $U=\bigchancenodemixture_{i=1}^n E_iU_i$, where $U_i\in\nfd(\tree_i)$, so
  \begin{equation*}
    \biggambplus_{i=1}^n E_i \normgambles(U_i) \subseteq \biggambplus_{i=1}^n E_i \normgambles(\normoper_{\opt}(\tree_i)).
  \end{equation*}

  Consider any $k$ and any normal form decision $V_k\in\normoper_{\opt}(\tree_k)$. The above equation, and Eq.~\eqref{eq:gambofnormoptisopt}, tell us that we can choose each $V_k$ such that $E_k\normgambles(V_k)=E_k\normgambles(U_k)$. Of course, because $V_k\in\normoper_{\opt}(\tree_k)$,
  \begin{equation*}
    \normgambles(V_k)\subseteq\opt(\normgambles(\tree_k)|\treeevent{\tree_k}).
  \end{equation*}
  We wish to establish that also $\normgambles(U_k)\subseteq\opt(\normgambles(\tree_k)|\treeevent{\tree_k})$.

  Indeed, this follows from Property~\ref{prop:conditioning:property}: observe that both singletons $\normgambles(U_k)$ and $\normgambles(V_k)$ are subsets of $\normgambles(\tree_k)$, $E_k\normgambles(U_k)=E_k\normgambles(V_k)$, and $\normgambles(V_k)\subseteq\opt(\normgambles(\tree_k)|\treeevent{\tree_k})$. Consistency of $\tree$ confirms that $\normgambles(\tree_k)$ is $\treeevent{\tree_k}$-consistent. Hence, Property~\ref{prop:conditioning:property} applies, and
  \begin{equation*}
    \normgambles(U_k)\subseteq\opt(\normgambles(\tree_k)|\treeevent{\tree_k}).
  \end{equation*}

  Therefore, $U_k\in\normoper_{\opt}(\tree_k)$ by definition of $\normoper_{\opt}(\tree_i)$. Since this holds for any $k$, we conclude that $U\in\bigchancenodemixture_{i=1}^nE_i\normoper_{\opt}(\tree_i)$. So, we have shown that also
  \begin{equation*}
    \normoper_{\opt}(\tree)\subseteq\bigchancenodemixture_{i=1}^nE_i\normoper_{\opt}(\tree_i).
  \end{equation*}
\end{proof}

\subsection{Proof of Lemma~\ref{lemma:gamble:normopt:equivalence:decision:nodes}}

\begin{proof}
  Assume that Eq.~\eqref{lemma:gamble:normopt:equivalence:decision:nodes:helper} holds.

  Consider any normal form decision $V\in\nfd\left(\bigdecnodeunion_{i\in\mathcal{I}} \normoper_{\opt}(\tree_i)\right)$. By definition of $\normgambles$,
  \begin{align*}
    \normgambles(V)&\subseteq \bigcup_{i\in\mathcal{I}}\normgambles(\normoper_{\opt}(\tree_i))
    \\
    \intertext{and, by Eq.~\eqref{lemma:gamble:normopt:equivalence:decision:nodes:helper},}
    &=\normgambles(\normoper_{\opt}(\tree)).
  \end{align*}
  Hence, by definition of $\normoper_{\opt}$, and the obvious fact that $V\in\nfd(T)$, it follows that $V\in\normoper_{\opt}(T)$.
  So we have shown that
  \begin{equation*}
    \normoper_{\opt}(\tree) \supseteq \nfd\Bigg(\bigdecnodeunion_{i\in\mathcal{I}}\normoper_{\opt}(\tree_i)\Bigg).
  \end{equation*}
  
  Conversely, let $V \in\normoper_{\opt}(\tree)$. Then, again by Eq.~\eqref{lemma:gamble:normopt:equivalence:decision:nodes:helper},
  \begin{equation*}
    \normgambles(V) \subseteq \normgambles(\normoper_{\opt}(\tree)) = \bigcup_{i\in\mathcal{I}}\normgambles(\normoper_{\opt}(\tree_i)).
  \end{equation*}
  Now $V=\decnodeunion U$ where $U\in\nfd(\tree_i)$ for some $i\in\mathcal{I}$. We want to show that $U\in\normoper_{\opt}(\tree_i)$.

  Indeed, let $X$ be the gamble corresponding to $V$, and also $U$,
  \begin{equation*}
    \normgambles(V)=\normgambles(U)=\{X\}.
  \end{equation*}
  Because $V \in\normoper_{\opt}(\tree)$, we know that $X\in\opt(\normgambles(\tree)|\treeevent{\tree})$. It is established that $U\in\normoper_{\opt}(\tree_i)$ if we can show that $X\in\opt(\normgambles(\tree_i)|\treeevent{\tree_i})$. But this follows at once from Property~\ref{prop:intersection:property}, because $i\in\mathcal{I}$ (and the definition of $\mathcal{I}$), $\normgambles(\tree_i)\subseteq\normgambles(\tree)$, $\treeevent{\tree}=\treeevent{\tree_i}$, and all sets of gambles are consistent with respect to the relevant events:
  \begin{equation*}
    \opt(\normgambles(\tree_i)|\treeevent{\tree_i})=\opt(\normgambles(\tree)|\treeevent{\tree})\cap\normgambles(\tree_i).
  \end{equation*}
  Concluding, also
  \begin{equation*}
    \normoper_{\opt}(\tree)\subseteq\nfd\Bigg(\bigdecnodeunion_{i\in\mathcal{I}}\normoper_{\opt}(\tree_i)\Bigg).
  \end{equation*}
\end{proof}

\subsection{Proof of Lemma~\ref{lemma:conditioning:intersection:mixture:are:necessary:for:factuality}}

\begin{proof}
  Assume that $\normoper_{\opt}$ is subtree perfect.

  We first establish Property~\ref{prop:conditioning:property}. Let $A$ be a non-empty event, and let $\mathcal{X}$ be a non-empty finite set of $A$-consistent gambles such that $\{X,Y\}\subseteq \mathcal{X}$ with $AX=AY$ and $X\in\opt(\mathcal{X}|A)$. We show that $Y\in\opt(\mathcal{X}|A)$. If $A=\Omega$ the result is trivial, so assume $A\subset\Omega$.
  
  Consider a consistent decision tree $\tree=A\tree_1 \chancenodemixture \compl{A}\tree_2$, where $\treeevent{\tree}=\Omega$, $\normgambles(\tree_1)=\mathcal{X}$, and $\tree_2$ is a normal form decision with $\normgambles(\tree_2)=\{Z\}$, an $\compl{A}$-consistent gamble (see the left tree in Fig.~\ref{fig:conditioning:intersection:mixture:are:necessary:for:factuality}). We know by consistency of the gambles that there is such a $\tree$ (see Definition~\ref{def:gambles:consistent}).
  
  Consider $U\in\nfd(\tree_1)$ with $\normgambles(U)=\{X\}$ and $V\in\nfd(\tree_1)$ with $\normgambles(V)=\{Y\}$. By definition of $\normoper_{\opt}$, we have $U\in\normoper_{\opt}(\tree_i)$. Therefore by subtree perfectness, $AU\chancenodemixture\compl{A}\tree_2\in\normoper_{\opt}(\tree)$ and of course $AV\chancenodemixture\compl{A}\tree_2\in\normoper_{\opt}(\tree)$. Again by subtree perfectness, $V\in\normoper_{\opt}(\tree_i)$, whence $Y\in\opt(\mathcal{X}|A)$.

  Next, we establish Property~\ref{prop:intersection:property}. Let $A$ be a non-empty event, and let $\mathcal{Y}\subseteq\mathcal{X}$ be non-empty finite $A$-consistent sets of gambles such that $\opt(\mathcal{X}|A)\cap\mathcal{Y}\ne \emptyset$. We show that $\opt(\mathcal{Y}|A) = \opt(\mathcal{X}|A)\cap\mathcal{Y}$.
  
  Let $\tree=\tree_1\decnodeunion\tree_2$ be a consistent decision tree with $\treeevent{\tree}=A$, $\normgambles(\tree_1)=\mathcal{Y}$ and $\normgambles(\tree_2)=\mathcal{X}$ (see the right tree in Fig.~\ref{fig:conditioning:intersection:mixture:are:necessary:for:factuality}). We know by consistency of the gambles that there is such a $\tree$.
  
  Let $\decnode$ be the decision node at the root of $\tree_1$. By subtree perfectness, we have
  \begin{equation*}
    \normgambles(\subtreeat{\normoper_{\opt}(\tree)}{N})=\normgambles(\normoper_{\opt}(\subtreeat{\tree}{N})).
  \end{equation*}
  The right-hand side is equal to $\opt(\mathcal{Y}|A)$. Also,
  \begin{align*}
    \normgambles(\subtreeat{\normoper_{\opt}(\tree)}{N})&=\normgambles(\normoper_{\opt}(\tree))\cap\normgambles(\subtreeat{\tree}{N}) \\
    &= \opt(\mathcal{X}|A)\cap\mathcal{Y}
  \end{align*}
  as required.

  Finally, we establish Property~\ref{prop:mixture:property}. Let $A$ and $B$ be events such that $A\cap B \ne \emptyset$ and $\compl{A}\cap B\ne \emptyset$, let $\mathcal{X}$ be a non-empty finite $A\cap B$-consistent set of gambles, and let $Z$ be a $\compl{A}\cap B$-consistent gamble.
  
  Let $\tree=A\tree_1\chancenodemixture\compl{A}\tree_2$ be a consistent decision tree such that $\treeevent{\tree}=B$, $\normgambles(\tree_1)=\mathcal{X}$, and $\normgambles(\tree_2)=\{Z\}$ (see the left tree in Fig.~\ref{fig:conditioning:intersection:mixture:are:necessary:for:factuality}).
  
  By subtree perfectness, we have (letting $N$ be the root node of $\tree_1$)
  \begin{equation*}
    \normgambles(\subtreeat{\normoper_{\opt}(\tree)}{N})=\normgambles(\normoper_{\opt}(\subtreeat{\tree}{N})).
  \end{equation*}
  Here, the right-hand side is $\opt(\mathcal{X}|A \cap B)$, and
  \begin{equation*}
    \normgambles(\subtreeat{\normoper_{\opt}(\tree)}{N})=\{X\in\mathcal{X}\colon AX\gambplus\compl{A}Z\in\opt(A\mathcal{X}\gambplus\compl{A}Z|B),
  \end{equation*}
  whence Property~\ref{prop:mixture:property} follows.
\end{proof}

\subsection{Proof of Theorem~\ref{thm:backopt:normopt:equivalence}}

Before we prove Theorem~\ref{thm:backopt:normopt:equivalence}, we prove a few lemmas, provide two counterexamples, and introduce a new property called path independence.

\begin{lemma}[Sen \protect{\cite[Proposition~17]{1977:sen}}]
  \label{lemma:opt:of:unions:subseteq}
  A choice function $\opt$ satisfies Property~\ref{prop:preservation:under:addition:of:elements} if and only if, for any non-empty event $A$ and any finite family of non-empty finite $A$-consistent sets of gambles $\mathcal{X}_1$, \dots $\mathcal{X}_n$,
  \begin{equation*}
    \opt\Bigg(\bigcup_{i=1}^n\mathcal{X}_i\Bigg|A\Bigg) \subseteq \opt\Bigg(\bigcup_{i=1}^n\opt(\mathcal{X}_i|A)\Bigg| A \Bigg)\subseteq \bigcup_{i=1}^n\opt(\mathcal{X}_i|A).
  \end{equation*}
\end{lemma}

\begin{property}\label{prop:path:independence}
 A choice function $\opt$ is said to be \emph{path independent} if, for any non-empty event $A$, and for any finite family of non-empty finite $A$-consistent sets of gambles $\mathcal{X}_1$, \dots, $\mathcal{X}_n$,
\begin{equation*}
 \opt\Bigg(\bigcup_{i=1}^n\mathcal{X}_i\Bigg|A\Bigg) = \opt\Bigg(\bigcup_{i=1}^n\opt(\mathcal{X}_i|A) \Bigg| A\Bigg).
\end{equation*}
\end{property}

Path independence appears frequently in the social choice literature. Plott \cite{1973:plott} gives a detailed investigation of path independence and its possible justifications. Path independence is also equivalent to Axiom $7'$ of Luce and Raiffa \cite[p.~289]{1957:luce}.

\begin{lemma}[Sen \protect{\cite[Proposition~19]{1977:sen}}]
  \label{lemma:opt:of:unions:equality}
  A choice function $\opt$ satisfies Properties~\ref{prop:insensitivity:to:non:optimal:elements} and~\ref{prop:preservation:under:addition:of:elements} if and only if $\opt$ satisfies Property~\ref{prop:path:independence}.
\end{lemma}

Properties~\ref{prop:insensitivity:to:non:optimal:elements}, \ref{prop:preservation:under:addition:of:elements}, and~\ref{prop:path:independence} are expressed slightly differently here than in Sen~\cite{1977:sen}, who does not use the concepts of conditioning and consistency. Despite this, the proofs of Lemmas~\ref{lemma:opt:of:unions:subseteq} and~\ref{lemma:opt:of:unions:equality} proceed identically to the corresponding propositions by Sen. Also, Sen defines path independence only for pairs of subsets, but Plott~\cite[Theorem~1, p.~1082]{1973:plott} shows that this type of path independence is indeed equivalent to Property~\ref{prop:path:independence}.

\begin{lemma}\label{lemma:setsums:subseteq}
  Let $A_1$, \dots, $A_n$ be a finite partition of $\pspace$, $n\ge 2$. Let $\mathcal{X}$, and $\mathcal{X}_1$, $\mathcal{Z}_1$, \dots, $\mathcal{X}_n$, $\mathcal{Z}_n$ be finite sets of gambles. If
  \begin{equation*}
    \mathcal{X}\subseteq A_k\mathcal{X}_k\gambplus\compl{A}_k \mathcal{Z}_k
  \end{equation*}
  for all $k\in\{1,\dots,n\}$, then
  \begin{equation*}
    \mathcal{X}\subseteq \biggambplus_{k=1}^nA_k\mathcal{X}_k.
  \end{equation*}
\end{lemma}

\begin{proof}
  Let $X\in\mathcal{X}$, then $X\in A_k \mathcal{X}_k \gambplus \compl{A}_k \mathcal{Z}_k$ for all $k$. Therefore, for each $k$, there is an $X_k\in\mathcal{X}_k$ such that $A_k X = A_k X_k$. Whence,
  \begin{equation*}
    X=\biggambplus_{k=1}^n A_k X_k \in \biggambplus_{k=1}^n A_k \mathcal{X}_k.
  \end{equation*}
\end{proof}

\begin{lemma}\label{lemma:opt:of:setsums:subseteq}
Let $A_1$, \dots, $A_n$ be a finite partition of $\pspace$. Let $B$ be any event such that $A_i \cap B \neq \emptyset$ for all $A_i$. Let $\mathcal{X}_1$, \dots, $\mathcal{X}_n$ be a finite family of non-empty finite sets of gambles where each $\mathcal{X}_i$ is $A_i \cap B$-consistent. If a choice function $\opt$ satisfies Properties~\ref{prop:preservation:under:addition:of:elements} and~\ref{prop:backward:mixture:property}, then
\begin{equation}\label{eq:opt:of:setsums:subseteq}
  \opt\Bigg(\biggambplus_{i=1}^nA_i\mathcal{X}_i\Bigg|B\Bigg) \subseteq \biggambplus_{i=1}^n A_i \opt(\mathcal{X}_i|A_i \cap B).
\end{equation}
\end{lemma}
\begin{proof}
Let $\mathcal{X}=\biggambplus_{i=1}^nA_i\mathcal{X}_i$. Consider any $k\in\{1,\dots,n\}$ and let $$\mathcal{Z}_k=\biggambplus_{j\neq k} A'_j \mathcal{X}_j$$
where $(A'_j)_{j\neq k}$ forms an arbitrary partition of $\pspace$ such that $\compl{A}_k\cap A'_j=A_j$ for all $j\neq k$. Clearly, $\mathcal{Z}_k$ is $\compl{A}_k\cap B$-consistent because we can trivially find a consistent decision tree $\tree$ with $\treeevent{\tree}=\compl{A}_k\cap B$ and $\normgambles(\tree)=\mathcal{Z}_k$, using the $A_j \cap B$-consistency (and hence, $\compl{A}_k\cap A'_j \cap B$-consistency) of each $\mathcal{X}_j$ for $j\neq k$.

Now, observe that by construction of $\mathcal{Z}_k$,
\begin{equation*}
\mathcal{X}=A_k \mathcal{X}_k \gambplus \compl{A}_k \mathcal{Z}_k
 =\bigcup_{Z_k\in\mathcal{Z}_k} (A_k \mathcal{X}_k \gambplus \compl{A}_k Z_k).
\end{equation*}
Note that $\mathcal{X}$ is $B$-consistent (indeed, because each $\mathcal{X}_i$ is $A_i \cap B$-consistent, we can trivially find a consistent decision tree $\tree$ with $\treeevent{\tree}=B$ and $\normgambles(\tree)=\mathcal{X}$).

If we apply $\opt(\cdot|B)$ on both sides of the above equality, then it follows from Lemma~\ref{lemma:opt:of:unions:subseteq} that
\begin{align*}
\opt(\mathcal{X}|B)&\subseteq \bigcup_{Z_k\in\mathcal{Z}_k}\opt\left(A_k \mathcal{X}_k \gambplus \compl{A}_k Z_k|B\right)
\\
\intertext{and by Property~\ref{prop:backward:mixture:property} (once noted that $\mathcal{X}_k$ is $A_k\cap B$-consistent by assumption, and $Z_k$ is $\compl{A}_k\cap B$-consistent by construction),}
&\subseteq \bigcup_{Z_k\in\mathcal{Z}_k} (A_k \opt(\mathcal{X}_k|A_k \cap B) \gambplus \compl{A}_k Z_k) \\
&=A_k \opt(\mathcal{X}_k | A_k \cap  B) \gambplus \compl{A}_k \mathcal{Z}_k \\
\intertext{whence by Lemma~\ref{lemma:setsums:subseteq},}
\opt(\mathcal{X}|B)&\subseteq \biggambplus_{i=1}^n A_i \opt(\mathcal{X}_i | A_i \cap B).
\end{align*}
\end{proof}

\begin{lemma}\label{lemma:opt:of:setsums:equality}
Let $A_1$, \dots, $A_n$ be a finite partition of $\pspace$. Let $B$ be any event such that $A_i \cap B \neq \emptyset$ for all $A_i$. Let $\mathcal{X}_1$, \dots, $\mathcal{X}_n$ be a finite family of non-empty finite sets of gambles, where each $\mathcal{X}_i$ is $A_i\cap B$-consistent. If a choice function $\opt$ satisfies Properties~\ref{prop:insensitivity:to:non:optimal:elements}, \ref{prop:preservation:under:addition:of:elements}, and~\ref{prop:backward:mixture:property}, then
\begin{equation}\label{eq:opt:of:setsums:equality}
  \opt\Bigg(\biggambplus_{i=1}^nA_i\mathcal{X}_i\Bigg|B\Bigg) = \opt\Bigg(\biggambplus_{i=1}^n A_i \opt(\mathcal{X}_i|A_i \cap B) \Bigg| B\Bigg).
\end{equation}
\end{lemma}
\begin{proof}
By Lemma~\ref{lemma:opt:of:setsums:subseteq} and the definition of $\opt$,
\begin{equation*}
 \opt\Bigg(\biggambplus_{i=1}^nA_i\mathcal{X}_i\Bigg|B\Bigg) \subseteq \biggambplus_{i=1}^n A_i \opt(\mathcal{X}_i|A_i \cap B)
 \subseteq \biggambplus_{i=1}^n A_i \mathcal{X}_i,
\end{equation*}
whence Eq.~\eqref{eq:opt:of:setsums:equality} follows by Property~\ref{prop:insensitivity:to:non:optimal:elements}.
\end{proof}

\begin{lemma}\label{lemma:gambnorm:norm:relationship:chancenodes}
  For any consistent decision tree $\tree=\bigchancenodemixture_{i=1}^n \event_i\tree_i$ and any choice function $\opt$ satisfying Property~\ref{prop:backward:conditioning:property},
  \begin{equation}\label{eq:lemma:gambnorm:norm:chancenodes:helper}
    \normgambles(\normoper_{\opt}(\tree)) = \normgambles\Bigg(\normoper_{\opt}\Bigg(\bigchancenodemixture_{i=1}^n \event_i \normoper_{\opt}(\tree_i)\Bigg)\Bigg)
  \end{equation}
  implies
  \begin{equation*}
    \normoper_{\opt}(\tree) = \normoper_{\opt}\Bigg(\bigchancenodemixture_{i=1}^n \event_i \normoper_{\opt}(\tree_i)\Bigg).
  \end{equation*}
\end{lemma}

\begin{proof}
  In this proof we require Eq.~\eqref{eq:gambofnormoptisopt}, and the following consequence of Eq.~\eqref{eq:normgambles:chancenodes:for:sets:of:trees}:
  \begin{multline}\label{eq:lemma:gambnorm:norm:chancenodes:helper1}
    \normgambles\Bigg(\normoper_{\opt}\Bigg(\bigchancenodemixture_{i=1}^n \event_i \normoper_{\opt}(\tree_i)\Bigg)\Bigg) \\ = \opt\Bigg(\biggambplus_{i=1}^n \event_i \opt(\normgambles(\tree_i) | \treeevent{\tree}\cap\event_i)\Bigg|\treeevent{\tree}\Bigg).
  \end{multline}
  We first show that
  \begin{equation*}
    \normoper_{\opt}(\tree) \supseteq \normoper_{\opt}\Bigg(\bigchancenodemixture_{i=1}^n \event_i \normoper_{\opt}(\tree_i)\Bigg).
  \end{equation*}
  
  Consider a normal form decision $\atree\in\normoper_{\opt}(\bigchancenodemixture_{i=1}^n \event_i \normoper_{\opt}(\tree_i))$. We have, by Eq.~\eqref{eq:lemma:gambnorm:norm:chancenodes:helper},
  \begin{equation*}
    \normgambles(\atree) \subseteq \normgambles\Bigg(\normoper_{\opt}\Bigg(\bigchancenodemixture_{i=1}^n \event_i \normoper_{\opt}(\tree_i)\Bigg)\Bigg) = \normgambles(\normoper_{\opt}(\tree)).
  \end{equation*}
  So, there exists a normal form decision $V\in\normoper_{\opt}(\tree)$ such that $\normgambles(V)=\normgambles(\atree)$. Since $\atree\in\nfd(\tree)$, by definition of $\normoper_{\opt}$, $\atree\in\normoper_{\opt}(\tree)$, which establishes the claim.
  
  Next, we show that
  \begin{equation*}
    \normoper_{\opt}(\tree) \subseteq \normoper_{\opt}\Bigg(\bigchancenodemixture_{i=1}^n \event_i \normoper_{\opt}(\tree_i)\Bigg).
  \end{equation*}
  Consider a normal form decision $\atree\in\normoper_{\opt}(\tree)$. We know by Eq.~\eqref{eq:lemma:gambnorm:norm:chancenodes:helper} and Eq.~\eqref{eq:lemma:gambnorm:norm:chancenodes:helper1} that
  \begin{equation*}
    \normgambles(\atree) \subseteq \opt\Bigg(\biggambplus_{i=1}^n \event_i \opt(\normgambles(\tree_i) | \treeevent{\tree}\cap\event_i)\Bigg| \treeevent{\tree}\Bigg).
  \end{equation*}
  We can write $\atree=\bigchancenodemixture_{i=1}^n \event_i \atree_i$, where $\atree_i \in\nfd(\tree_i)$, so by Eq.~\eqref{eq:normgambles:chancenodes},
  \begin{equation*}
    \biggambplus_{i=1}^n \event_i \normgambles(\atree_i) \subseteq \opt\Bigg(\biggambplus_{i=1}^n \event_i \opt(\normgambles(\tree_i)| \treeevent{\tree}\cap\event_i)\Bigg|\treeevent{\tree}\Bigg).
  \end{equation*}
  Consider normal form decisions $V_i \in\normoper_{\opt}(\tree_i)$. The above equation, and Eq.~\eqref{eq:gambofnormoptisopt} tell us that, for each $i$, we can find $V_i$ such that $\event_i \normgambles(V_i) = \event_i \normgambles(\atree_i)$. Of course, because $V_i \in \normoper_{\opt}(\tree_i)$,
  \begin{equation*}
    \normgambles(V_i) \subseteq \opt(\normgambles(\tree_i) | \event_i \cap \treeevent{\tree}).
  \end{equation*}
  We further have, for $V=\bigchancenodemixture_{i=1}^n \event_i V_i$, $\normgambles(V)=\normgambles(\atree)$ and $V\in\nfd(\tree)$, and so $V\in\normoper_{\opt}(\tree)$.
  
  If we can establish that
  \begin{equation}\label{eq:lemma:gambnorm:norm:chancenodes:helper2}
   \normgambles(\atree_i) \subseteq \opt(\normgambles(\tree_i) | \event_i \cap \treeevent{\tree}),
   \end{equation}
   then, by definition of $\normoper_{\opt}$ and because $\atree_i \in \nfd(\tree_i)$, it follows that $\atree_i \in \normoper_{\opt}(\tree_i)$. So, in that case, there is a $V\in\normoper_{\opt}(\bigchancenodemixture_{i=1}^n \event_i \normoper_{\opt}(\tree_i))$ such that $\normgambles(V)=\normgambles(U)$, and $U\in\nfd(\bigchancenodemixture_{i=1}^n \event_i \normoper_{\opt}(\tree_i))$. Therefore by definition of $\normoper_{\opt}$, we will have $U\in\normoper_{\opt}(\bigchancenodemixture_{i=1}^n \event_i \normoper_{\opt}(\tree_i))$, establishing the desired result.
  
  We show that Eq.~\eqref{eq:lemma:gambnorm:norm:chancenodes:helper2} indeed holds by Property~\ref{prop:backward:conditioning:property}. When $n=1$ the result is trivial, so assume $n\geq2$. Observe that both singletons $\normgambles(\atree_i)$ and $\normgambles(V_i)$ are subsets of $\normgambles(\tree_i)$, $\event_i \normgambles(\atree_i) = \event_i \normgambles(V_i)$, and $\normgambles(V_i) \subseteq \opt(\normgambles(\tree_i) | \event_i \cap \treeevent{\tree})$. Further, $\normgambles(\tree_i)$ is $\event_i \cap \treeevent{\tree}$-consistent. We are almost ready to apply Property~\ref{prop:backward:conditioning:property}.
  
  We know that
  \begin{align*}
    \normgambles(V) = \biggambplus_{i=1}^n \event_i \normgambles(V_i) &\subseteq \opt(\normgambles(\tree) | \treeevent{\tree}) \\
    &= \opt\Bigg(\biggambplus_{i=1}^n \event_i \normgambles(\tree_i) \Bigg| \treeevent{\tree}\Bigg).
  \end{align*}
  Letting
  \begin{equation*}
    Z = (\event_1 \cup \event_2) \normgambles(V_2) \gambplus \event_3 \normgambles(V_3) \gambplus \dots \gambplus \event_n \normgambles(V_n)
  \end{equation*}
  and
  \begin{equation*}
    \mathcal{Z} = (\event_1 \cup \event_2) \normgambles(\tree_2) \gambplus \event_3 \normgambles(\tree_3) \gambplus \dots \gambplus \event_n \normgambles(\tree_n),
  \end{equation*}
  we see that $Z \in\mathcal{Z}$ and $\mathcal{Z}$ is $\compl{\event}_1 \cap \treeevent{\tree}$-consistent. Further,
  \begin{equation*}
    \event_1 \normgambles(V_1) \gambplus \compl{\event}_1 Z = \biggambplus_{i=1}^n \event_i \normgambles(V_i) = \normgambles(V)
  \end{equation*}
  and
  \begin{equation*}
        \event_1 \normgambles(\tree_1) \gambplus \compl{\event}_1 \mathcal{Z} = \biggambplus_{i=1}^n \event_i \normgambles(\tree_i).
  \end{equation*}
  We see that
  \begin{equation*}
    E_1 \normgambles(V) \gambplus \compl{E}_1 Z \subseteq \opt(E_1 \normgambles(\tree_1) \gambplus \compl{E}_1 \mathcal{Z} | \treeevent{\tree}).
  \end{equation*}
  Hence we have found a $\mathcal{Z}$ and a $Z\in\mathcal{Z}$ required to apply Property~\ref{prop:backward:conditioning:property}. Finally, by $E_1 \normgambles(\atree_1) = E_1 \normgambles(V_1)$, and Property~\ref{prop:backward:conditioning:property}, we have
  \begin{equation*}
    \normgambles(\atree_1) \subseteq \opt(\normgambles(\tree_1) | \event_1 \cap \treeevent{\tree}).
  \end{equation*}
  This argument applies for any index $i$ and therefore Eq.~\eqref{eq:lemma:gambnorm:norm:chancenodes:helper2} has been shown, establishing the result.
\end{proof}

\begin{lemma}\label{lemma:gambnorm:norm:relationship:decnodes}
	For any consistent decision tree $\tree=\bigdecnodeunion_{i=1}^n \tree_i$, and any choice function $\opt$ satisfying Property~\ref{prop:preservation:under:addition:of:elements},
	\begin{equation}\label{eq:lemma:gambnorm:norm:decnodes:helper}
		\normgambles(\normoper_{\opt}(\tree)) = \normgambles\Bigg(\normoper_{\opt}\Bigg(\bigdecnodeunion_{i=1}^n \normoper_{\opt}(\tree_i)\Bigg)\Bigg)
	\end{equation}
	implies
	\begin{equation*}
		\normoper_{\opt}(\tree) = \normoper_{\opt}\Bigg(\bigdecnodeunion_{i=1}^n \normoper_{\opt}(\tree_i)\Bigg).
	\end{equation*}
\end{lemma}

\begin{proof}
  In this proof we require of Eq.~\eqref{eq:gambofnormoptisopt}. For clarity, let $A=\treeevent{\tree}=\treeevent{\tree_i}$. We first show that
  \begin{equation*}
    \normoper_{\opt}(\tree) \supseteq \normoper_{\opt}\Bigg(\bigdecnodeunion_{i=1}^n \normoper_{\opt}(\tree_i)\Bigg).
  \end{equation*}
   Consider a normal form decision $\atree\in\normoper_{\opt}(\bigdecnodeunion_{i=1}^n \normoper_{\opt}(\tree_i))$. To show that $\atree\in\normoper_{\opt}(\tree)$, we must show that $\atree\in\nfd(\tree)$ and $\normgambles(\atree)\subseteq\normgambles(\normoper_{\opt}(\tree))$. The former is obvious, and the latter is established by Eq.~\eqref{eq:lemma:gambnorm:norm:decnodes:helper}:
  \begin{equation*}
    \normgambles(\atree) \subseteq \normgambles\Bigg(\normoper_{\opt}\Bigg(\bigdecnodeunion_{i=1}^n \normoper_{\opt}(\tree_i)\Bigg)\Bigg) = \normgambles(\normoper_{\opt}(\tree)).
  \end{equation*}
  
  Next we show that  
   \begin{equation*}
    \normoper_{\opt}(\tree) \subseteq \normoper_{\opt}\Bigg(\bigdecnodeunion_{i=1}^n \normoper_{\opt}(\tree_i)\Bigg).
  \end{equation*} 
  Let $\atree \in\normoper_{\opt}(\tree)$. To show that $\atree\in\normoper_{\opt}(\bigdecnodeunion_{i=1}^n \normoper_{\opt}(\tree_i))$ we must show that $\atree\in\nfd(\bigdecnodeunion_{i=1}^n \normoper_{\opt}(\tree_i))$ and that
  \begin{equation*}
    \normgambles(\atree)\subseteq\normgambles\Bigg(\normoper_{\opt}\Bigg(\bigdecnodeunion_{i=1}^n \normoper_{\opt}(\tree_i)\Bigg)\Bigg).
  \end{equation*}
  The latter requirement follows immediately from Eq.~\eqref{eq:lemma:gambnorm:norm:decnodes:helper}:
  \begin{equation*}
    \normgambles(\atree) \subseteq \normgambles(\normoper_{\opt}(\tree)) = \normgambles\Bigg(\normoper_{\opt}\Bigg(\bigdecnodeunion_{i=1}^n \normoper_{\opt}(\tree_i)\Bigg)\Bigg).
  \end{equation*}
  
  We now prove that $\atree\in\nfd(\bigdecnodeunion_{i=1}^n \normoper_{\opt}(\tree_i))$. Let $V$ be $\atree$ with the root node removed, that is, $U=\decnodeunion V$. Clearly, for some $k$, $V\in\nfd(\tree_k)$. It suffices to show that $V\in\normoper_{\opt}(\tree_k)$. Let $\{X\}=\normgambles(\atree)=\normgambles(V)$. We know that $X\in\opt(\normgambles(\tree)|A)$, and also that $X\in\normgambles(\tree_k)$. If we can prove that $X\in\normgambles(\normoper_{\opt}(\tree_k))=\opt(\normgambles(\tree_k) | A)$, then $V\in\normoper_{\opt}(\tree_k)$.
  Indeed, using Property~\ref{prop:preservation:under:addition:of:elements} and $\normgambles(\tree_k)\subseteq\normgambles(\tree)$ we have
  \begin{equation*}
    \opt(\normgambles(\tree_k) | A) \supseteq \normgambles(\tree_k) \cap \opt(\normgambles(\tree) | A).
  \end{equation*}
  So we have shown that indeed $X\in\opt(\normgambles(\tree_k) | A)$, establishing the claim.
\end{proof}

\begin{lemma}\label{lemma:backward:conditioning:property:is:necessary}
	If $\backopt(\tree)=\normoper_{\opt}(\tree)$ for any consistent decision tree $\tree$, then $\opt$ satisfies Property~\ref{prop:backward:conditioning:property}.
\end{lemma}

\begin{proof}
  Let $A$ and $B$ be non-empty events, and $\mathcal{X}$, $\mathcal{Z}$ be non-empty finite sets of gambles, such that the following properties hold: $A\cap B\neq \emptyset$, $\compl{A} \cap B \neq \emptyset$, $\mathcal{X}$ is $A\cap B$-consistent, $\mathcal{Z}$ is $\compl{A} \cap B$-consistent, and there are $X,Y\in\mathcal{X}$ such that $AX=AY$, $X \in\opt(\mathcal{X} | A\cap B)$, and $AX \gambplus \compl{A}Z \in\opt(A\mathcal{X}\gambplus \compl{A}\mathcal{Z} | B)$ for at least one $Z\in\mathcal{Z}$. If it is not possible to construct such a situation, then $\opt$ satisfies Property~\ref{prop:backward:conditioning:property} automatically. Otherwise, to prove that Property~\ref{prop:backward:conditioning:property} holds, we must show that $Y\in\opt(\mathcal{X}|A\cap B)$.
  
  Consider a consistent decision tree $\tree=A\tree_1 \chancenodemixture \compl{A}\tree_2$, where $\treeevent{\tree}=B$, $\normgambles(\tree_1)=\mathcal{X}$, and $\normgambles(\tree_2)=\mathcal{Z}$. Since $\mathcal{X}$ is $A\cap B$-consistent and $\mathcal{Z}$ is $\compl{A}\cap B$-consistent, we know from Definition~\ref{def:gambles:consistent} that there is such a $\tree$. We have $\normgambles(\normoper_{\opt}(\tree)) = \opt(\normgambles(\tree) | B)=\opt(A\mathcal{X} \gambplus \compl{A} \mathcal{Z} | B)$. So, $AX\gambplus\compl{A}Z\in\normgambles(\normoper_{\opt}(\tree))$, and of course $AX\gambplus\compl{A}Z=AY \gambplus\compl{A}Z$.
	
  Therefore, any normal form decision in $\nfd(\tree)$ that induces the gamble $AY \gambplus \compl{A} Z$ must be in $\normoper_{\opt}(\tree)$. In particular, by Lemma~\ref{lemma:gambisgambnfd} there is a normal form decision $U\in\nfd(\tree_1)$ such that $\normgambles(U)=\{Y\}$, and a normal form decision $V\in\nfd(\tree_2)$ such that $\normgambles(V)=\{Z\}$. So $AU\chancenodemixture\compl{A}V\in\nfd(\tree)$ and $\normgambles(AU \chancenodemixture \compl{A}V)=\{AY\gambplus\compl{A}Z\}$. Indeed, because $AX\gambplus\compl{A}Z\in\opt(A\mathcal{X}\gambplus\compl{A}\mathcal{Z}|B)$, it follows that $AU\chancenodemixture\compl{A}V\in\normoper_{\opt}(\tree)=\backopt(\tree)$. By definition, $\backopt(\tree) = \normoper_{\opt}(A\backopt(\tree_1) \chancenodemixture \compl{A}\backopt(\tree_2))$, and so it must hold that $U\in\backopt(\tree_1) = \normoper_{\opt}(\tree_1)$. Whence, $\normgambles(U) \subseteq \normgambles(\normoper_{\opt}(\tree_i)) = \opt(\mathcal{X} | A \cap B)$. Since $\normgambles(U)=\{Y\}$, we have $Y\in\opt(\mathcal{X} | A \cap B)$, establishing Property~\ref{prop:backward:conditioning:property}.
\end{proof}

\begin{lemma}\label{lemma:path:independence:is:necessary}
    If $\normgambles(\backopt(\tree))=\normgambles(\normoper_{\opt}(\tree))$ for any consistent decision tree $\tree$, then $\opt$ satisfies Property~\ref{prop:path:independence}.
\end{lemma}

\begin{proof}
  Let $A$ be a non-empty event, and $\mathcal{X}_1$, \dots, $\mathcal{X}_n$ be non-empty finite sets of $A$-consistent gambles. Let $\tree=\bigdecnodeunion_{i=1}^n \tree_i$ be a consistent decision tree, with $\normgambles(\tree_i) = \mathcal{X}_i$ for each $i$, and where $\treeevent{\tree}=A$. The existence of $\tree$ is assured by $A$-consistency of $\mathcal{X}$ (see Definition~\ref{def:gambles:consistent}). We have
  \begin{align*}
    \opt\Bigg(\bigcup_{i=1}^n \mathcal{X}_i\Bigg|A\Bigg) &= \normgambles(\normoper_{\opt}(\tree)) = \normgambles(\backopt(\tree)) \\
    &= \normgambles\Bigg(\normoper_{\opt}\Bigg(\bigdecnodeunion_{i=1}^n \backopt(\tree_i)\Bigg)\Bigg)\\
  \intertext{From Eq.~\eqref{eq:gambofnormoptisopt}, we have $\normgambles(\normoper_{\opt}(\tree))=\opt(\normgambles(\tree) | A)$. Similarly, with repeated applications of Eq.~\eqref{eq:gambofnormoptisopt} and~Eq.~\eqref{eq:normgambles:decisionnodes:for:sets:of:trees},}
    &= \opt\Bigg(\normgambles\Bigg(\bigdecnodeunion_{i=1}^n \backopt(\tree_i)) \Bigg| A\Bigg) \\
    &= \opt\Bigg(\bigcup_{i=1}^n \normgambles(\backopt(\tree_i)) \Bigg| A\Bigg) \\
    &= \opt\Bigg(\bigcup_{i=1}^n \normgambles(\normoper_{\opt}(\tree_i)) \Bigg| A\Bigg) \\
    &= \opt\Bigg(\bigcup_{i=1}^n \opt(\normgambles(\tree_i) | A) \Bigg| A\Bigg).
  \intertext{Finally, we note that $\normgambles(\tree)=\bigcup_{i=1}^n \mathcal{X}_i$ and $\normgambles(\tree_i)=\mathcal{X}_i$. Therefore,}
    &= \opt\Bigg(\bigcup_{i=1}^n \opt(\mathcal{X}_i | A) \Bigg| A\Bigg).
  \end{align*}
\end{proof}

\begin{lemma}\label{lemma:backward:mixture:property:is:necessary}
  If $\normgambles(\backopt(\tree))=\normgambles(\normoper_{\opt}(\tree))$ for any consistent decision tree $\tree$, then $\opt$ satisfies Property~\ref{prop:backward:mixture:property}.
\end{lemma}

\begin{proof}
  Let $A$ and $B$ be non-empty events such that $A\cap B \neq \emptyset$ and $\compl{A} \cap B \neq \emptyset$, let $\mathcal{X}$ be a non-empty finite set of $A\cap B$-consistent gambles, and let $Z$ be an $\compl{A} \cap B$-consistent gamble. Let $\tree=A\tree_1 \chancenodemixture \compl{A}\tree_2$ be a consistent decision tree, where $\treeevent{\tree}=B$, $\normgambles(\tree_1) = \mathcal{X}$, and $\normgambles(\tree_2)=\{Z\}$. The existence of $\tree$ is assured by $A\cap B$-consistency of $\mathcal{X}$ and $\compl{A}\cap B$-consistency of $\{Z\}$.
By assumption,
  \begin{align*}
    \normgambles(\normoper_{\opt}(\tree)) &= \normgambles(\backopt(\tree)) \\
    &= \normgambles(\normoper_{\opt}(A\backopt(\tree_1) \chancenodemixture \compl{A} \tree_2)).
  \end{align*}
  
  From Eq.~\eqref{eq:gambofnormoptisopt}, we have $\normgambles(\normoper_{\opt}(\tree)) = \opt(\normgambles(\tree) | B)$. Similarly, with repeated applications of Eq.~\eqref{eq:gambofnormoptisopt} and~Eq.~\eqref{eq:normgambles:chancenodes:for:sets:of:trees},
  \begin{align*}
    \normgambles(\normoper_{\opt}(A\backopt(\tree_1) \chancenodemixture \compl{A} \tree_2)) &= \opt(\normgambles(A\backopt(\tree_1) \chancenodemixture \compl{A} \tree_2) | B ) \\
    &= \opt(A\normgambles(\backopt(\tree_1)) \gambplus \compl{A}Z | B) \\
    &= \opt(A\normgambles(\normoper_{\opt}(\tree_1)) \gambplus \compl{A}Z | B) \\
    &= \opt(A\opt(\normgambles(\tree_1)|A\cap B) \gambplus \compl{A}Z | B).
  \end{align*}
  Finally we note that $\normgambles(\tree) = A\mathcal{X} \gambplus \compl{A}Z$, and $\normgambles(\tree_1) = \mathcal{X}$. Therefore,
  \begin{equation*}
    \opt(A\mathcal{X} \gambplus \compl{A}Z | B)=\opt(A\opt(\mathcal{X} | A\cap B) \gambplus \compl{A}Z | B) \subseteq A\opt(\mathcal{X} | A\cap B) \gambplus \compl{A}Z.
  \end{equation*}
\end{proof}

We now prove Theorem~\ref{thm:backopt:normopt:equivalence}.

\begin{proof}[Proof of Theorem~\ref{thm:backopt:normopt:equivalence}]
  ``only if''. By Lemmas~\ref{lemma:backward:conditioning:property:is:necessary}, \ref{lemma:path:independence:is:necessary}, and~\ref{lemma:backward:mixture:property:is:necessary}, we see that satisfying backward induction implies Properties~\ref{prop:backward:conditioning:property}, \ref{prop:backward:mixture:property}, and~\ref{prop:path:independence}. Lemma~\ref{lemma:opt:of:unions:equality} completes the proof.

  ``if''. We prove this part by structural induction on the tree. In the base step, we prove that the implication holds for consistent decision trees which consist of only a single node. In the induction step, we prove that if the implication holds for the subtrees at every child of the root node, then the implication also holds for the whole tree.

First, if the decision tree $\tree$ has only a single node, and hence, a reward at the root and no further children, then by definition (Eq.~\eqref{eq:normgambles:rewards} in particular) we have $\backopt(\tree)=\normoper_{\opt}(\tree)$.

Next, suppose $\tree$ is consistent and has a chance node as its root, that is, $\tree=\bigchancenodemixture_{i=1}^n \event_i \tree_i$. By the induction hypothesis, we know that for every $\tree_i$,
\begin{equation}\label{eq:backopt:normopt:equivalence:induction:assumption}
 \normgambles(\backopt(\tree_i)) = \normgambles(\normoper_{\opt}(\tree_i)).
\end{equation}
We show that $\backopt(\tree)=\normoper_{\opt}(\tree)$. 
By Lemma~\ref{lemma:gambnorm:norm:relationship:chancenodes}, it therefore suffices to show that $\normgambles(\normoper_{\opt}(\tree)) = \normgambles(\backopt(\tree))$. By Eq.~\eqref{eq:gambofnormoptisopt} and the definition of $\normgambles$,
\begin{align*}
	\normgambles(\normoper_{\opt}(\tree)) &= \opt(\normgambles(\tree) | \treeevent{\tree})\\
	&= \opt\Bigg(\normgambles\Bigg(\bigchancenodemixture_{i=1}^n \event_i \tree_i\Bigg) \Bigg| \treeevent{\tree}\Bigg)\\
	&= \opt\Bigg(\biggambplus_{i=1}^n \event_i \normgambles(\tree_i)\Bigg| \treeevent{\tree}\Bigg),
\end{align*}
and by Eq.~\eqref{eq:backopt:normopt:equivalence:induction:assumption}, Eq.~\eqref{eq:gambofnormoptisopt}, and the definition of $\normgambles$,
\begin{align*}
	\normgambles(\backopt(\tree)) &= \normgambles\Bigg(\normoper_{\opt}\Bigg(\bigchancenodemixture_{i=1}^n \event_i \backopt(\tree_i)\Bigg)\Bigg) \\
	&= \opt\Bigg(\normgambles\Bigg(\bigchancenodemixture_{i=1}^n \event_i \backopt(\tree_i)\Bigg)\Bigg|\treeevent{\tree}\Bigg) \\
	&= \opt\Bigg(\biggambplus_{i=1}^n \event_i\normgambles(\backopt(\tree_i)) \Bigg| \treeevent{\tree}\Bigg) \\
	&= \opt\Bigg(\biggambplus_{i=1}^n \event_i\normgambles(\normoper_{\opt}(\tree_i)) \Bigg| \treeevent{\tree}\Bigg) \\
	&= \opt\Bigg(\biggambplus_{i=1}^n \event_i \opt(\normgambles(\tree_i) | \treeevent{\tree} \cap \event_i) \Bigg| \treeevent{\tree}\Bigg),
\end{align*}
whence equality follows from Lemma~\ref{lemma:opt:of:setsums:equality}.

Finally, suppose that the root of the consistent tree $\tree$ is a decision node, that is $\tree=\bigdecnodeunion_{i=1}^n \tree_i$. 
We show that $\backopt(\tree)=\normoper_{\opt}(\tree)$. By Lemma~\ref{lemma:gambnorm:norm:relationship:decnodes}, it suffices to show that $\normgambles(\backopt(\tree))=\normgambles(\normoper_{\opt}(\tree))$. Indeed,
\begin{align*}
	\normgambles(\normoper_{\opt}(\tree)) &= \opt(\normgambles(\tree)|\treeevent{\tree})\\
	&= \opt\Bigg(\normgambles\Bigg(\bigdecnodeunion_{i=1}^n\tree_i\Bigg) \Bigg| \treeevent{\tree}\Bigg)\\
	&= \opt\Bigg(\bigcup_{i=1}^n \normgambles(\tree_i) \Bigg| \treeevent{\tree}\Bigg),
\end{align*}
and,
\begin{align*}
  \normgambles\Bigg(\normoper_{\opt}\Bigg(\bigdecnodeunion_{i=1}^n \backopt(\tree_i)\Bigg)\Bigg) &= \opt\Bigg(\normgambles\Bigg(\bigdecnodeunion_{i=1}^n \backopt(\tree_i)\Bigg)\Bigg| \treeevent{\tree}\Bigg) \\
    &= \opt\Bigg(\bigcup_{i=1}^n \normgambles(\backopt(\tree_i))\Bigg| \treeevent{\tree}\Bigg) \\
  &= \opt\Bigg(\bigcup_{i=1}^n \normgambles(\normoper_{\opt}(\tree_i))\Bigg| \treeevent{\tree}\Bigg) \\
  &= \opt\Bigg(\bigcup_{i=1}^n \opt(\normgambles(\tree_i) | \treeevent{\tree_i}) \Bigg| \treeevent{\tree}\Bigg) \\
    &= \opt\Bigg(\bigcup_{i=1}^n \opt(\normgambles(\tree_i) | \treeevent{\tree}) \Bigg| \treeevent{\tree}\Bigg),
\end{align*}
whence equality follows by Lemma~\ref{lemma:opt:of:unions:equality} and Property~\ref{prop:path:independence}.

Concluding, we have shown that the implication holds for consistent decision trees consisting of a single nodes, and that if the implication holds for all children of the root node then it also holds for the whole tree. By induction, the implication holds for any consistent decision tree.
\end{proof}

\subsection{Proof of Lemma~\ref{lem:intersection:implies:insensitivity:and:preservation}}

\begin{proof}
  Assume Property~\ref{prop:intersection:property} holds.
  
  Consider any non-empty finite sets of normal form decisions $\mathcal{X}$ and $\mathcal{Y}$, and any event $A\neq\emptyset$, such that $\opt(\mathcal{X}|A) \subseteq \mathcal{Y} \subseteq \mathcal{X}$. By Property~\ref{prop:intersection:property}, it follows that $\opt(\mathcal{Y}|A)=\opt(\mathcal{X}|A)\cap\mathcal{Y}$, which is equal to $\opt(\mathcal{X}|A)$ because $\opt(\mathcal{X}|A) \subseteq \mathcal{Y}$. This proves Property~\ref{prop:insensitivity:to:non:optimal:elements}.

  Consider any non-empty finite sets of normal form decisions $\mathcal{X}$ and $\mathcal{Y}$, and any event $A\neq\emptyset$ such that $\mathcal{Y} \subseteq \mathcal{X}$. If $\opt(\mathcal{X}|A)\cap\mathcal{Y}=\emptyset$, then obviously $\opt(\mathcal{Y}|A) \supseteq \opt(\mathcal{X}|A) \cap \mathcal{Y}$. If not, then $\opt(\mathcal{Y}|A) = \opt(\mathcal{X}|A)\cap\mathcal{Y}$. So, Property~\ref{prop:preservation:under:addition:of:elements} follows.
\end{proof}

\subsection{Proof of Theorem~\ref{thm:weak:subtree:perfectness}}

\begin{lemma}\label{lemma:weak:perfectness:at:children:implies:weak:perfectness:everywhere}
  Let $\normoper$ be any normal form operator. Let $\tree$ be a consistent decision tree. If,
  \begin{enumerate}[label=(\roman*)]
  \item\label{lemma:weak:perfectness:at:children:childcondition} for all nodes $K\in\children(\tree)$ such that $K$ is in at least one element of $\normoper(\tree)$,
    \begin{equation*}
      \subtreeat{\normoper(\tree)}{K}\subseteq\normoper(\subtreeat{\tree}{K}),
    \end{equation*}
  \item\label{lemma:weak:perfectness:at:children:grandchildcondition} and, for all nodes $K\in\children(\tree)$, and all nodes $L\in\subtreeat{\tree}{K}$ such that $L$ is in at least one element of $\normoper(\subtreeat{\tree}{K})$,
  \begin{equation*}
    \subtreeat{\normoper(\subtreeat{\tree}{K})}{L}\subseteq\normoper(\subtreeat{\subtreeat{\tree}{K}}{L}),
  \end{equation*}
  \end{enumerate}
  then, for all nodes $N$ in $\tree$ such that $N$ is in at least one element of $\normoper(\tree)$,
  \begin{equation*}
    \subtreeat{\normoper(\tree)}{N}\subseteq\normoper(\subtreeat{\tree}{N}).
  \end{equation*}
\end{lemma}

\begin{proof}
  If $N$ is the root of $\tree$, then the result is immediate. If $N\in\children(\tree)$, then the result follows from \ref{lemma:weak:perfectness:at:children:childcondition}. Otherwise, $N$ must be in $\subtreeat{\tree}{K}$ for one $K\in\children(\tree)$.

  By assumption, there is a $U\in\normoper(\tree)$ that contains $N$ (and of course also $K$). Therefore, $U\in\subtreeat{\normoper(\tree)}{K}$, and by \ref{lemma:weak:perfectness:at:children:childcondition}, $\subtreeat{U}{K}\in\normoper(\subtreeat{\tree}{K})$, and so $N$ is also in at least one element of $\normoper(\subtreeat{\tree}{K})$.

  We use the fact that, if $\mathcal{U}$ and $\mathcal{V}$ are sets of normal form decisions such that $\mathcal{U}\subseteq\mathcal{V}$, then for any node $N$, $\subtreeat{\mathcal{U}}{N}\subseteq\subtreeat{\mathcal{V}}{N}$. Combining everything, by \ref{lemma:weak:perfectness:at:children:childcondition},
  \begin{align*}
    \subtreeat{\subtreeat{\normoper(\tree)}{K}}{N} &\subseteq \subtreeat{\normoper(\subtreeat{\tree}{K})}{N} \\
    \intertext{hence, since $N$ is in at least one element of $\normoper(\subtreeat{\tree}{K})$, by \ref{lemma:weak:perfectness:at:children:grandchildcondition} we have}
    &\subseteq \normoper(\subtreeat{\subtreeat{\tree}{K}}{N}),
  \end{align*}
  whence the desired result follows, since $\subtreeat{\subtreeat{T}{K}}{N}=\subtreeat{T}{N}$.
\end{proof}

The following results are very similar to Lemmas~\ref{lemma:backward:conditioning:property:is:necessary}, \ref{lemma:path:independence:is:necessary}, and~\ref{lemma:backward:mixture:property:is:necessary}.

\begin{lemma}\label{lemma:weak:subtree:perfectness:backward:conditioning:property:is:necessary}
  If $\normoper_{\opt}$ is subtree perfect for normal form decisions, then $\opt$ satisfies Property~\ref{prop:backward:conditioning:property}.
\end{lemma}

\begin{proof}
  Let $A$ and $B$ be non-empty events, and $\mathcal{X}$, $\mathcal{Z}$ be non-empty finite sets of gambles, such that the following properties hold: $A\cap B\neq \emptyset$, $\compl{A} \cap B \neq \emptyset$, $\mathcal{X}$ is $A\cap B$-consistent, $\mathcal{Z}$ is $\compl{A} \cap B$-consistent, and there are $X,Y\in\mathcal{X}$ such that $AX=AY$, $X \in\opt(\mathcal{X} | A\cap B)$, and $AX \gambplus \compl{A}Z \in\opt(A\mathcal{X}\gambplus \compl{A}\mathcal{Z} | B)$ for at least one $Z\in\mathcal{Z}$. If it is not possible to construct such a situation, then $\opt$ satisfies Property~\ref{prop:backward:conditioning:property} automatically. Otherwise, to prove that Property~\ref{prop:backward:conditioning:property} holds, we must show that $Y\in\opt(\mathcal{X}|A\cap B)$.
  
  Consider a consistent decision tree $\tree=A\tree_1 \chancenodemixture \compl{A}\tree_2$, where $\treeevent{\tree}=B$, $\normgambles(\tree_1)=\mathcal{X}$, and $\normgambles(\tree_2)=\mathcal{Z}$. Since $\mathcal{X}$ is $A\cap B$-consistent and $\mathcal{Z}$ is $\compl{A}\cap B$-consistent, we know from Definition~\ref{def:gambles:consistent} that there is such a $\tree$. By Lemma~\ref{lemma:gambisgambnfd}, there is a normal form decision $U\in\nfd(\tree_1)$ such that $\normgambles(U)=\{Y\}$, and a normal form decision in $V\in\nfd(\tree_2)$ such that $\normgambles(V)=\{Z\}$.
  
  Since, by assumption, $AX=AY$, obviously $AX \gambplus \compl{A}Z=AY \gambplus \compl{A}Z$, and hence, also $AY \gambplus \compl{A}Z\in\opt(A\mathcal{X}\gambplus \compl{A}\mathcal{Z} | B)$. Therefore, by definition of $\normoper_{\opt}$, $AU\chancenodemixture\compl{A}V \in \normoper_{\opt}(\tree)$. In particular, $U\in\subtreeat{\normoper_{\opt}(\tree)}{N}$, where $N$ is the root of $\tree_1$. Because $\normoper_{\opt}$ is subtree perfect for normal form decisions, it follows that also $U\in\normoper_{\opt}(\subtreeat{\tree}{N})=\normoper_{\opt}(\tree_1)$. Again applying the definition of $\normoper_{\opt}$, we conclude that indeed $Y\in\opt(\mathcal{X}|A\cap B)$.
\end{proof}

\begin{lemma}\label{lemma:weak:subtree:perfectness:preservation:is:necessary}
  If $\normoper_{\opt}$ is subtree perfect for normal form decisions, then $\opt$ satisfies Property~\ref{prop:preservation:under:addition:of:elements}.
\end{lemma}

\begin{proof}
  Let $A$ be a non-empty event, and let $\mathcal{X}$ be a non-empty finite set of $A$-consistent gambles. Let $\mathcal{Y}$ be a non-empty subset of $\mathcal{X}$. Let $\tree=\tree_1 \decnodeunion \tree_2$, where $\treeevent{\tree}=A$, $\normgambles(\tree_1)=\mathcal{Y}$, and $\normgambles(\tree_2)=\mathcal{X}$. Let $N$ be the root of $\tree_1$. 
  
  If $\opt(\mathcal{X}|A)\cap \mathcal{Y}=\emptyset$ then Property~\ref{prop:preservation:under:addition:of:elements} holds automatically. Suppose $\opt(\mathcal{X}|A)\cap\mathcal{Y}\neq\emptyset$. By definition of $\normoper_{\opt}$, $N$ appears in at least one element of $\normoper_{\opt}(\tree)$, and
  \begin{align*}
    \opt(\mathcal{X}|A)\cap \mathcal{Y}&=\normgambles(\subtreeat{\normoper_{\opt}(\tree)}{N}), \\
    \intertext{and by subtree perfectness of normal form decisions}
    &\subseteq \normgambles(\normoper_{\opt}(\tree_1)) \\
    &= \opt(\mathcal{Y}|A).
  \end{align*}

\end{proof}

\begin{lemma}\label{lemma:weak:subtree:perfectness:backward:mixture:is:necessary}
  If $\normoper_{\opt}$ is subtree perfect for normal form decisions, then $\opt$ satisfies Property~\ref{prop:backward:mixture:property}.
\end{lemma}

\begin{proof}
  Let $A$ and $B$ be non-empty events such that $A\cap B \neq \emptyset$ and $\compl{A} \cap B \neq \emptyset$, let $\mathcal{X}$ be a non-empty finite set of $A\cap B$-consistent gambles, and let $Z$ be an $\compl{A} \cap B$-consistent gamble. Let $\tree=A\tree_1 \chancenodemixture \compl{A}\tree_2$ be a consistent decision tree, where $\treeevent{\tree}=B$, $\normgambles(\tree_1) = \mathcal{X}$, and $\tree_2$ is simply a normal form decision with $\normgambles(\tree_2)=\{Z\}$. The existence of $\tree$ is assured by $A\cap B$-consistency of $\mathcal{X}$ and $\compl{A}\cap B$-consistency of $\{Z\}$. Let $N$ be the root of $\tree_1$.
  

  Consider any gamble $AX\gambplus\compl{A}Z\in\opt(A\mathcal{X}\gambplus\compl{A}Z|B)$. By Lemma~\ref{lemma:gambisgambnfd}, there is a $U\in\nfd(\tree_1)$ such that $\normgambles(U)=X$. By definition of $\normoper_{\opt}$, it follows that $AU\chancenodemixture\compl{A}\tree_2 \in \normoper_{\opt}(\tree)$, and hence, in particular, $U\in\subtreeat{\normoper_{\opt}(\tree)}{N}$. By subtree perfectness for normal form decisions, $U\in\normoper_{\opt}(\tree_1)$. Again applying the definition of $\normoper_{\opt}$, we find that $X\in\opt(\mathcal{X}|A\cap B)$, thus indeed $AX\gambplus\compl{A}Z\in A\opt(\mathcal{X}|A\cap B)\gambplus\compl{A}Z$, whence Property~\ref{prop:backward:mixture:property} is established.
\end{proof}

We can now prove Theorem~\ref{thm:weak:subtree:perfectness}.

\begin{proof}[Proof of Theorem~\ref{thm:weak:subtree:perfectness}]
  ``only if''. Follows from Lemmas~\ref{lemma:weak:subtree:perfectness:backward:conditioning:property:is:necessary}, \ref{lemma:weak:subtree:perfectness:preservation:is:necessary}, and~\ref{lemma:weak:subtree:perfectness:backward:mixture:is:necessary}.
  
  ``if''. We proceed as usual by structural induction. The base step is trivial as usual. Let $\children(\tree)=\{K_1,\dots,K_n\}$ and let $\tree_i=\subtreeat{\tree}{K_i}$. The induction hypothesis says that $\normoper_{\opt}$ is subtree perfect for normal form decisions on all $T_i$. More precisely, for all $\tree_i$, and for every $L$ that is in at least one element of $\normoper_{\opt}(\tree_i)$,
  \begin{equation*}
    \subtreeat{\normoper_{\opt}(\tree_i}{L}\subseteq \normoper_{\opt}(\subtreeat{\tree}{L}).
  \end{equation*}
  We must show that, for any $N$ in at least one element of $\normoper_{\opt}(\tree)$,
    \begin{equation*}
    \subtreeat{\normoper(\tree)}{N} \subseteq \normoper(\subtreeat{\tree}{N}).
  \end{equation*}
  By the induction hypothesis and Lemma~\ref{lemma:weak:perfectness:at:children:implies:weak:perfectness:everywhere}, it suffices to show this only for $N\in\children(\tree)$, that is, to show that
  \begin{equation}\label{eq:thm:weak:subtree:perfectness}
    \subtreeat{\normoper_{\opt}(\tree)}{K_i}\subseteq \normoper_{\opt}(\tree_i)
  \end{equation}
  for each $i$ such that $K_i$ is in at least one element of $\normoper_{\opt}(\tree)$.

   Suppose the root of $\tree$ is a decision node, so $\tree=\bigdecnodeunion_{i=1}^n \tree_i$. Let $U$ be an element of $\normoper_{\opt}(\tree)$. There is a $j$ such that $K_j$ is in $U$; let $U_j$ denote $\subtreeat{U}{K_j}$. To establish Eq.~\eqref{eq:thm:weak:subtree:perfectness} we must show that $U_j\in\normoper_{\opt}(\tree_j)$.
  
  Note that $\normgambles(U_j)=\normgambles(U)\subseteq\opt(\normgambles(\tree)|\treeevent{\tree})$ since $U\in\normoper_{\opt}(T)$. Obviously, also $\normgambles(U_j)\in\normgambles(\tree_j)$ by definition of $\normgambles$. Hence, it must hold that
  \begin{align*}
    \normgambles(U_j)&\subseteq\opt(\normgambles(\tree)|\treeevent{\tree})\cap\normgambles(\tree_j), \\
    \intertext{but, also, because $\normgambles(\tree_j)\subseteq\normgambles(\tree)$, and once noted that $\treeevent{\tree}=\treeevent{\tree_j}$, it follows from Property~\ref{prop:preservation:under:addition:of:elements} that}
    &\subseteq
    \opt(\normgambles(\tree_j)|\treeevent{\tree_j})
  \end{align*}
  Putting everything together, we confirm that $U_j\in\normoper_{\opt}(\tree_j)$. This proves the induction step for decision nodes.

  Now suppose that the root of $\tree$ is a chance node, so $\tree=\bigchancenodemixture_{i=1}^n \event_i \tree_i$. Again, let $U=\bigchancenodemixture_{i=1}^n \event_iU_i\in\normoper_{\opt}(\tree)$. To establish Eq.~\eqref{eq:thm:weak:subtree:perfectness} we must show that $U_i\in\normoper_{\opt}(\tree_i)$ for all $i$.

  Indeed, since $U\in\normoper_{\opt}(\tree)$,
  \begin{align*}
    \normgambles(U)\in\opt(\normgambles(T))
    &=\opt\Bigg(\biggambplus\event_i\normgambles(\tree_i)\Bigg|\treeevent{\tree}\Bigg)
    \\
    \intertext{so by Lemma~\ref{lemma:opt:of:setsums:subseteq},}
    &\subseteq \biggambplus\event_i \opt(\normgambles(\tree_i)|\treeevent{\tree}\cap\event_i).
  \end{align*}
  So, for each $\tree_i$, there is a normal form decision $V_i\in\normoper_{\opt}(\tree_i)$ such that $E_i \normgambles(V_i) = E_i \normgambles(U_i)$. Can we apply Property~\ref{prop:backward:conditioning:property}?
   
  Obviously, $\{\normgambles(V_i), \normgambles(U_i)\}\subseteq \normgambles(\tree_i)$, and 
  \begin{multline*}
    \event_i\normgambles(V_i)\gambplus\compl{\event_i}Z
    =\normgambles(U)
    \\
    \subseteq\opt(\biggambplus\event_i\normgambles(\tree_i)|\treeevent{\tree})
    =\opt(\event_i\normgambles(\tree_i)\gambplus\compl{\event_i}\mathcal{Z}|\treeevent{\tree})
  \end{multline*}
  for suitable choices for $\mathcal{Z}$ and a $Z\in\mathcal{Z}$. Therefore we can apply Property~\ref{prop:backward:conditioning:property} to conclude that $U_i$ is in $\normoper_{\opt}(T_i)$ for each $i$. This proves the induction step for chance nodes.
\end{proof}

\subsection{Proofs for Section~\ref{sec:equivalence}}

\begin{proof}[Proof of Lemma~\ref{lemma:equivalent:extensive:form:construction}]
  If $N$ is in at least one element of $\normoper(\tree)$ but does not appear in $\extoper(\tree)$, then $N$ will not appear in any element of $\nfd(\extoper(\tree))$ and equivalence will fail. If a node $M$ is in $\extoper(\tree)$ but not in any element of $\normoper(\tree)$, then $M$ will appear in at least one element of $\nfd(\extoper(\tree))$ and equivalence will fail. Therefore a node is in $\extoper(\tree)$ if and only if it is in at least one element of $\normoper(\tree)$.
\end{proof}

\begin{proof}[Proof of Lemma~\ref{lemma:corresponding:operators:both:factual:or:neither}]
  ``if''. Suppose $\extoper$ is subtree perfect and a node $N$ is in $\extoper(\tree)$, then $N$ is in at least one element of $\normoper(\tree)$. Because $\nfd$ and $\mathrm{st}_N$ commute, we have
  \begin{align*}
    \normoper(\subtreeat{\tree}{N}) &= \nfd(\extoper(\subtreeat{\tree}{N}))=\nfd(\subtreeat{\extoper(\tree)}{N}) \\
     &= \subtreeat{\nfd(\extoper(\tree))}{N}=\subtreeat{\normoper(\tree)}{N}.
  \end{align*}
  This demonstrates subtree perfectness of $\normoper$.
  
  ``only if''. By Lemma~\ref{lemma:equivalent:extensive:form:construction}, for a particular $\normoper$ there can be no more than one equivalent $\extoper$. We show that this $\extoper$ is subtree perfect. A node $N$ is in this $\extoper$ if and only if it is in at least one element of $\normoper(\tree)$. Similarly, a node $M$ is in $\subtreeat{\extoper(\tree)}{N}$ if and only if $M$ is in at least one element of $\subtreeat{\normoper(\tree)}{N}$. By subtree perfectness of $\normoper$, the latter is satisfied if and only if $M$ is in at least one element of $\normoper(\subtreeat{\tree}{N})$. But, again by definition of $\extoper$, the latter is satisfied if and only if $M$ is in $\extoper(\subtreeat{\tree}{N})$. This establishes subtree perfectness of $\extoper$.
\end{proof}

\begin{proof}[Proof of Theorem~\ref{thm:factual:norm:implies:corresponding:ext}]
  By Lemma~\ref{lemma:corresponding:operators:both:factual:or:neither}, if an equivalent $\extoper$ exists then it is subtree perfect. We must show that the $\extoper$ constructed in Lemma~\ref{lemma:equivalent:extensive:form:construction} satisfies
  \begin{equation*}
    \normoper_{\opt}(\tree)=\nfd(\extoper(\tree))
  \end{equation*}
  for all consistent decision trees $\tree$.
  
  We now proceed by structural induction. The base step, that $\nfd(\extoper(\tree))=\normoper_{\opt}(\tree)$ for any decision tree comprising only a single node, is as usual satisfied trivially.

  Let us proceed with the induction step. The induction hypothesis states that, for any node $K$ in $\children(\tree)$, $\nfd(\extoper(\subtreeat{\tree}{K}))=\normoper_{\opt}(\subtreeat{\tree}{K})$. We must show that $\normoper_{\opt}(\tree)=\nfd(\extoper(\tree))$. 
  
  It is useful to show first that, if $\mathcal{K}$ is the set of all $K\in\children(\tree)$ that appear in at least one element of $\normoper_{\opt}(\tree)$ (or equivalently, that appear in $\extoper(\tree)$), then for any $K\in\mathcal{K}$,
    \begin{equation}\label{thm:factual:norm:helper}
    \subtreeat{\nfd(\extoper(\tree))}{K} = \subtreeat{\normoper_{\opt}(\tree)}{K}.
  \end{equation} 
  Consider any node $K$ in $\children(\tree)$ that appears in $\extoper(\tree)$. Clearly,
  \begin{align*}
    \subtreeat{\nfd(\extoper(\tree))}{K}
    &=\nfd(\subtreeat{\extoper(\tree)}{K})
    \\
    \intertext{and since we just proved that $\extoper$ is subtree perfect,}
    &=\nfd(\extoper(\subtreeat{\tree}{K}))
    \\
    \intertext{and by the induction hypothesis,}
    &=\normoper_{\opt}(\subtreeat{\tree}{K})
    \\
    \intertext{but, by definition of $\extoper$, the node $K$ also appears in at least one element of $\normoper_{\opt}(\tree)$, so by the subtree perfectness of $\normoper_{\opt}$,} 
    &=\subtreeat{\normoper_{\opt}(\tree)}{K}.
  \end{align*}
  This establishes Eq.~\eqref{thm:factual:norm:helper}.
  
  Now, suppose that the root of $\tree$ is a decision node. Observe that \begin{align*}
    \nfd(\extoper(\tree)) &= \bigdecnodeunion_{K\in\mathcal{K}}\nfd(\subtreeat{\extoper(\tree)}{K}), \\
    \intertext{and since $\subtreeat{\cdot}{}$ and $\nfd(\cdot)$ commute,}
    &= \bigdecnodeunion_{K\in\mathcal{K}}\subtreeat{\nfd(\extoper(\tree))}{K}.
  \end{align*}
  Also, because $\opt$ satisfies Properties~\ref{prop:conditioning:property} and~\ref{prop:intersection:property}, we have (as seen in the proof of Theorem~\ref{thm:conditions:to:be:factual})
  \begin{equation*}
    \normoper_{\opt}(\tree)=\bigdecnodeunion_{K\in\mathcal{K}}\subtreeat{\normoper_{\opt}(\tree)}{K},
  \end{equation*}
  whence by Eq.~\eqref{thm:factual:norm:helper},
  \begin{equation*}
    \normoper_{\opt}(\tree)=\nfd(\extoper(\tree)).
  \end{equation*}
  
  Finally, suppose that the root of $\tree$ is a chance node. Here, $\mathcal{K}$ is simply $\children(\tree)=\{K_1, \dots, K_n\}$. Similarly to before, we have \begin{align*}
    \nfd(\extoper(\tree)) &= \bigchancenodemixture_{i=1}^n E_i \nfd(\subtreeat{\extoper(\tree)}{K_i}) \\
    &= \bigchancenodemixture_{i=1}^n E_i \subtreeat{\nfd(\extoper(\tree))}{K_i}.
  \end{align*} Since $\opt$ satisfies Properties~\ref{prop:conditioning:property}, \ref{prop:intersection:property}, and~\ref{prop:mixture:property}, we have (as seen in the proof of Theorem~\ref{thm:conditions:to:be:factual}),
   \begin{align*}
    \normoper_{\opt}(\tree)&=\bigchancenodemixture_{i=1}^n E_i \normoper_{\opt}(\subtreeat{\tree}{K_i}) \\
    &= \bigchancenodemixture_{i=1}^n E_i \subtreeat{\normoper_{\opt}(\tree)}{K_i},
  \end{align*}
  whence by Eq.~\eqref{thm:factual:norm:helper}, we have
  \begin{equation*}
    \nfd(\extoper(\tree))=\normoper_{\opt}(\tree).
  \end{equation*}
\end{proof}

\end{document}